\documentclass[a4paper,10pt,reqno]{amsart}
\usepackage[utf8x]{inputenc}
\usepackage{amsmath,amsthm,amssymb}
\usepackage{amsaddr}
\usepackage{tikz}
\usetikzlibrary{patterns}

\def\MR#1{}

\makeatletter
\let\@fnsymbol\@arabic
\makeatother 

\newtheorem{thm}{Theorem}[section]
\newtheorem{lem}[thm]{Lemma}
\newtheorem{prop}[thm]{Proposition}
\newtheorem{cor}[thm]{Corollary}

\theoremstyle{definition}
\newtheorem{definition}[thm]{Definition}
\newtheorem{remark}[thm]{Remark}
\newtheorem{example}[thm]{Example}

\newcommand{\FF}{\mathbb{F}}

\newcommand{\cG}{\mathcal{G}}
\newcommand{\cY}{\mathcal{Y}}

\newcommand{\EE}{\mathbb{E}}

\newcommand{\Mj}{\ensuremath{M_j}}

\newcommand{\Mjmin}{\ensuremath{M_j^-}}
\newcommand{\Mjst}{\ensuremath{M_j^*}}

\newcommand{\pj}{\ensuremath{p_j}}
\newcommand{\pjmom}{\ensuremath{p_{j-1}^-}}
\newcommand{\pjone}{\ensuremath{p_j^{(1)}}}

\newcommand{\pjmin}{\ensuremath{p_j^{-}}}

\newcommand{\pMj}{\ensuremath{p_{M_j}}}
\newcommand{\pMjmin}{\ensuremath{p_{M_{j-1}}}}
\newcommand{\pstMj}{\ensuremath{\bar{p}_{M_j}}}

\newcommand{\const}{\ensuremath{c}}
\newcommand{\pconn}{\ensuremath{p_{\text{conn}}}}
\newcommand{\pisol}{\ensuremath{p_{\text{isol}}}}

\newcommand{\varMjst}{\ensuremath{X_*}}
\newcommand{\varMjmin}{\ensuremath{X_-}}

\newcommand{\formalconnected}{$\FF_2$-cohomologically $j$-connected}
\newcommand{\connected}{$j$-cohom-connected}
\newcommand{\formalconnectedness}{$\FF_2$-cohomological $j$-connectedness}
\newcommand{\connectedness}{$j$-cohom-connectedness}

\newcommand{\mwconnectedness}{$\FF_2$-cohomological $(k-1)$-connectedness}
\newcommand{\mwconnected}{$\FF_2$-cohomologically $(k-1)$-connected}

\renewcommand{\Pr}{\mathbb{P}}

\renewcommand{\theenumi}{(\roman{enumi})}

\DeclareMathOperator{\im}{im}

\title{Vanishing of cohomology groups of random simplicial complexes}
\author{Oliver Cooley, Nicola Del Giudice, Mihyun Kang, Philipp Spr{\" u}ssel}
\address{Institute of Discrete Mathematics, Graz University of Technology, Steyrergasse 30, 8010 Graz, Austria}
\email{\{cooley,delgiudice,kang,spruessel\}@math.tugraz.at}
\thanks{Supported by Austrian Science Fund (FWF): P27290 and W1230 II}

\keywords{Random hypergraphs, random simplicial complexes, sharp threshold, hitting time, connectedness}

\begin{document}

\begin{abstract}
  We consider $k$-dimensional random simplicial complexes that are
  generated from the binomial random $(k+1)$-uniform hypergraph by
  taking the downward-closure, where $k\geq 2$. For each $1\leq j \leq
  k-1$, we determine when all cohomology groups with coefficients in
  $\FF_2$ from dimension one up to $j$ vanish and the zero-th
  cohomology group is isomorphic to $\FF_2$. This property is not
  deterministically monotone for this model of random complexes, 
  but nevertheless we show that it has a
  single sharp threshold. Moreover we prove a hitting time result,
  relating the vanishing of these cohomology groups to the
  disappearance of the last minimal obstruction. We also study the
  asymptotic distribution of the dimension of the $j$-th cohomology
  group inside the critical window. As a corollary, we deduce a
  hitting time result for a different model of random simplicial
  complexes introduced in [Linial and Meshulam, Combinatorica, 2006],
  a result which was previously only known for dimension two [Kahle
  and Pittel, Random Structures Algorithms, 2016].
\end{abstract}

\maketitle

\section{Introduction}

\subsection{Motivation}

In their seminal paper~\cite{ErdosRenyi59}, Erd\H{o}s and R\'enyi 
introduced the uniform random graph and, among other results,
addressed the problem of determining the probability of this graph
being connected. This classical result is usually stated for
the binomial random graph $G(n,p)$ on $n$ vertices,
in which each edge is present with a given
probability $p$ independently: the property of $G(n,p)$ being connected
undergoes a \emph{phase transition} around 
the sharp threshold $p=\frac{\log n }{n}$~\cite{Stepanov70}.
Throughout the paper, we denote the natural logarithm by $\log$ and we say 
that an event holds \emph{with high probability} (\emph{whp}
for short) if it holds with probability tending to $1$
as $n$ tends to infinity.
\begin{thm}[\cite{ErdosRenyi59,Stepanov70}]\label{thm:ER}
  Let $\omega$ be any function of $n$ which tends to infinity as
  $n \to \infty$. Then \emph{with high probability} 
  the following holds.
  \begin{enumerate}
  \item If $p=\frac{\log n - \omega}{n}$, then $G(n,p)$ is not
    connected.
  \item If $p=\frac{\log n + \omega}{n}$, then $G(n,p)$ is connected.
  \end{enumerate}
\end{thm} 
\noindent
As an even stronger result, Erd\H{o}s and R\'enyi~\cite{ErdosRenyi59}
determined the limiting probability of $G(n,p)$ being connected
around the point
of the phase transition. More precisely, this result can be stated for
$G(n,p)$ as follows.

\begin{thm}[see
  e.g.~{\cite[Theorem~4.1]{FriezeKaronski16}}] \label{thm:ERcritical}
  Let $\const \in \mathbb{R}$ be a constant and suppose that
  $(\const_n)_{n\geq 1}$ is a sequence of real numbers that converges to
  $\const$ as $n \rightarrow \infty$ . If 
  \begin{equation*}
    p = \frac{\log n +  \const_n}{n},
  \end{equation*}
  then
  \begin{equation*}
    \Pr \left( G(n,p) \mbox{ is connected} \right)
    \xrightarrow{n \rightarrow \infty } e^{-e^{-\const}}.
  \end{equation*}
\end{thm}
We note that while~\cite[Theorem~4.1]{FriezeKaronski16} 
is stated for the uniform random graph, it is actually proved 
via the binomial model $G(n,p)$ and thus immediately translates into
Theorem~\ref{thm:ERcritical}.

Subsequently, Bollob\'as and Thomason~\cite{BollobasThomason85} proved
a \emph{hitting time} result for the random graph process, in which
edges are added one at a time uniformly at random. This result relates
the connectedness of the random graph process to the disappearance of
the last \emph{smallest obstruction},  an isolated vertex.

\begin{thm}[\cite{BollobasThomason85}]\label{thm:BT}
  With high probability, the random graph process becomes connected at
  exactly the moment when the last isolated vertex disappears.
\end{thm}

Since then, many higher-dimensional analogues of both random graphs
and connectedness have been analysed and in particular two different
approaches have received considerable attention. A first natural
generalisation for dimension $k\geq 1$ is the random $(k+1)$-uniform
hypergraph $G_p=G(k;n,p)$ in which each $(k+1)$-tuple of vertices
forms a hyperedge with probability $p$ independently. There are
several natural ways of defining connectedness of $G_p$, which have
been extensively studied, including vertex-connectedness
\cite{BehrischCojaOghlanKang10b,BCOK14,BollobasRiordan12c,BollobasRiordan17,KaronskiLuczak96,Poole15,SchmidtShamir85}
and high-order connectedness (also known as $j$-tuple-connectedness)
\cite{CooleyKangKoch16,CKKgiant,CooleyKangPerson18,KahlePittel16}.
Another topic which has received particular attention is
generalisations of the $\ell$-core of a random graph (i.e.\ the
maximum subgraph with minimum degree at least $\ell$)
\cite{BotelhoWormaldZiviani12,Cooper04,DemboMontanari08,Molloy05},
which itself may be viewed as a generalisation of the giant component
of a random graph
\cite{Bollobas84,ErdosRenyi60,JansonKnuthLuczakPittel93,Luczak90,LuczakPittelWierman}.

A more recent approach concerns random simplicial complexes, of which
a first model for the 2-dimensional case was introduced by Linial and
Meshulam~\cite{LinialMeshulam06}. They considered the concept of
\emph{$\FF_2$-homological $1$-connectivity} of the random 2-complex as
the vanishing of its first homology group with coefficients in the
two-element field $\FF_2$, which is equivalent to the vanishing of the
first cohomology group. More precisely, the model $\cY_p=\cY(k;n,p)$
considered by Linial and Meshulam~\cite{LinialMeshulam06} for $k=2$
and subsequently by Meshulam and Wallach~\cite{MeshulamWallach08} for
general $k\ge2$ is defined as follows. Starting from the full
\emph{$(k-1)$-dimensional skeleton} on $[n]:=\{1,\ldots,n\}$, that is,
all simplices from dimension zero up to $k-1$, each $(k+1)$-set forms
a $k$-simplex with probability $p$ independently. They showed that the
property of the vanishing of the $(k-1)$-th cohomology 
group $H^{k-1}(\cY_p;\FF_2)$
with coefficients in $\FF_2$ has a sharp threshold at $p =
\frac{k \log n}{n}$.

\begin{thm}[\cite{LinialMeshulam06,MeshulamWallach08}]\label{thm:meshwall}
  Let $\omega$ be any function of $n$ which tends to infinity as $n
  \rightarrow \infty$. Then with high probability,
  \begin{enumerate}
  \item if $p=\frac{k\log n - \omega}{n}$, then
    $H^{k-1}(\cY_p;\FF_2) \neq 0$;
  \item if $p=\frac{k\log n + \omega}{n}$, then
    $H^{k-1}(\cY_p;\FF_2) = 0$.
  \end{enumerate}
\end{thm}
Meshulam and Wallach~\cite{MeshulamWallach08} further proved that the same
statement remains true if the coefficients of the cohomology group are
taken from any finite abelian group.

Later, Kahle and Pittel~\cite{KahlePittel16} derived a hitting time
result for $\cY_p$ (analogous to Theorem~\ref{thm:BT}) in the case
$k=2$. Moreover, they determined the limiting distribution of
$\dim\left(H^{k-1}(\cY_p;\FF_2)\right)$ for general $k\geq 2$ and 
for $p$ inside the critical window.

\begin{thm}[{\cite[Theorem~1.10]{KahlePittel16}}]\label{thm:KPcritical}
  Let $k\geq 2$ and $\const \in \mathbb{R}$ be a constant. If 
  \begin{equation*}
    p = \frac{k \log n +  \const}{n},
  \end{equation*}
  then $\dim\left(H^{k-1}(\cY_p;\FF_2)\right)$ converges in
  distribution to a Poisson random variable with expectation
  $e^{-\const}/k!$. In particular, we have
  \begin{equation*}
    \Pr \left( H^{k-1}(\cY_p;\FF_2) = 0 \right)
    \xrightarrow{n \rightarrow \infty} e^{-e^{-\const}/k!}.
  \end{equation*}
\end{thm}

Observe that Theorem~\ref{thm:KPcritical} can be generalised to hold
for $p = (k\log n+\const_n)/n$, where $(\const_n)_{n\geq 1}$ is a
sequence of real numbers that converges to $\const$ as $n \rightarrow 
\infty$ (cf.\ Theorem~\ref{thm:ERcritical}), because
$\dim\left(H^{k-1}(\cY_p;\FF_2)\right)$ is a monotone function in $p$.

In this paper, we aim to bridge the gap between random hypergraphs and
random simplicial complexes, considering random simplicial
$k$-complexes that arise as the downward-closure of random
$(k+1)$-uniform hypergraphs (Definition~\ref{def:complexGp}). 
Unlike $\cY_p$, in this model the presence of 
the full $(k-1)$-dimensional skeleton is not guaranteed,
thus the vanishing of the cohomology groups of dimensions lower than
$k-1$ does not hold trivially. Therefore, for each $1\leq j\leq k-1$,
we introduce \emph{\formalconnectedness} of a $k$-dimensional
simplicial complex (Definition~\ref{def:cohomconn}) as the vanishing
of \emph{all} cohomology groups with coefficients in $\FF_2$ from
dimension one up to $j$ and the zero-th cohomology group being
isomorphic to $\FF_2$.

Although this notion of connectedness is not deterministically
monotone for our model, we prove that \formalconnectedness\ has a sharp
threshold. Furthermore, we derive a hitting time result and determine
the limiting probability for \formalconnectedness\ inside the critical
window. As a corollary, we deduce a hitting time result for $\cY_p$ in
\emph{general dimension}, thus extending the hitting time result of
Kahle and Pittel~\cite{KahlePittel16}.

\subsection{Model}

Throughout the paper let $k\geq 2$ be a fixed integer. For positive
integers $\ell$ and $1\le i\le\ell$, write $[\ell]:=\{1,\dotsc,\ell\}$
and denote by $\binom{[\ell]}{i}$ the family of $i$-element subsets of
$[\ell]$.
\begin{definition}
  A family $\cG$ of non-empty finite subsets of a vertex set $V$ is
  called a \emph{simplicial complex} if it is downward-closed, i.e.\
  if every non-empty set $A$ that is contained in a set $B\in \cG$
  also lies in $\cG$, and if furthermore the singleton $\{v\}$ is in
  $\cG$ for every $v\in V$.
  
  The elements of a simplicial complex $\cG$ of cardinality $k+1$ are
  called \emph{$k$-simplices} of $\cG$. If $\cG$ has no
  $(k+1)$-simplices, then we call it \emph{$k$-dimensional}, or
  \emph{$k$-complex}. If $\cG$ is a $k$-complex, then for each
  $j=0,\ldots,k-1$ the \emph{$j$-skeleton} of $\cG$ is the $j$-complex
  formed by all $i$-simplices in $\cG$ with $0\leq i\leq j$.
\end{definition}

We define a model of random $k$-complexes starting from the binomial
random $(k+1)$-uniform hypergraph $G_p$ on vertex set $[n]$: the
$0$-simplices are the vertices of $G_p$, the $k$-simplices are the
hyperedges of $G_p$, but there is more than one way to guarantee the
downward-closure property to obtain a simplicial complex. In the model
$\cY_p$ considered by Meshulam and Wallach
in~\cite{MeshulamWallach08}, the full $(k-1)$-skeleton on $[n]$ is
always included. In contrast, we only include those simplices
that are \emph{necessary} to ensure the downward-closure property. 

\begin{definition} \label{def:complexGp}
  We denote by $\cG_p=\cG(k;n,p)$ the random $k$-dimensional
  simplicial complex on vertex set $[n]$ such that:
  \begin{itemize}
  \item the $0$-simplices are the singletons of $[n]$;
  \item the $k$-simplices are the hyperedges of the binomial random
    $(k+1)$-uniform hypergraph $G_p$;
  \item for each $j \in [k-1]$, the $j$-simplices are exactly the
    $(j+1)$-subsets of hyperedges of $G_p$.
  \end{itemize}
\end{definition}
\noindent
In other words, $\cG_p$ is the random $k$-complex on $[n]$ obtained
from $G_p$ by taking the downward-closure of each hyperedge. For
instance, denote by $F_p$ the set of hyperedges of the binomial random
4-uniform hypergraph $G_p = G(3;n,p)$. Then the corresponding two
models of random 3-dimensional simplicial complexes are given by
\begin{align*}
  \cY_p &= \cY(3;n,p) =  \binom{[n]}{1}\, \cup\, \binom{[n]}{2}\,
  \cup\, \binom{[n]}{3}\, \cup\, F_p\qquad \hspace{2ex} \text{and}\\
  \cG_p &= \cG(3;n,p) = \binom{[n]}{1}\, \cup\,
  \partial(\partial F_p)\, \cup\, \partial F_p\, \cup\, F_p,
\end{align*}
where $\partial E$ for a set $E$ of $j$-simplices, $j\ge1$, denotes
the set of all $(j-1)$-simplices that are contained in elements of
$E$.

Given a simplicial complex $\cG$, let $H^{i}(\cG;\FF_2)$ be its $i$-th
cohomology group with coefficients in $\FF_2$
(see~\eqref{eq:cohomgroup} in
Section~\ref{sec:preliminaries:cohomology} for the definition).
We define a notion of connectedness for a simplicial complex via
the \emph{vanishing} of its cohomology groups. Since the $0$-th cohomology 
group $H^{0}(\cG;\FF_2)$ cannot vanish, we require 
this group to be ``as small as possible''.
\begin{definition} \label{def:cohomconn}
  Given a positive integer $j$, a simplicial complex $\cG$ is called
  \emph{\formalconnected} (\emph{\connected} for short) if
  \begin{itemize}
  \item $H^0(\cG;\FF_2) = \FF_2$;
  \item $H^{i}(\cG;\FF_2)= 0$ for all $i \in [j]$.
  \end{itemize}
\end{definition}
Observe that $H^0(\cG;\FF_2)$ being isomorphic to $\FF_2$ is equivalent
to connectedness of $\cG$ in the topological sense, which we call
\emph{topological connectedness} in order to distinguish it from other
notions of connectedness. For $\cG = \cG_p$, this is also equivalent to
\emph{vertex-connectedness} of the associated $(k+1)$-uniform
hypergraph.

Moreover, one might define an analogous version of connectedness via the
vanishing of \emph{homology} groups, which would be equivalent to our 
definition of \formalconnectedness\ by the Universal Coefficient Theorem (see
e.g.~\cite{Munkres84}).

A significant difference between $\cG_p$ and $\cY_p$ is that for
$\cY_p$ the only requirement for \mwconnectedness\ is the vanishing of
the $(k-1)$-th cohomology group, since the presence of the full
$(k-1)$-skeleton guarantees topological connectedness and the
vanishing of the $j$-th cohomology groups for all $j \in [k-2]$. 

Moreover, it is important to observe that \formalconnectedness\ is
\emph{not} necessarily a monotone increasing property of $\cG_p$:\
adding a $k$-simplex to a \connected\ complex might yield a complex
without this property (see Example~\ref{ex:nonmono}). Thus, the
existence of a single threshold for \connectedness\ is \emph{not}
guaranteed, but one of our main results shows that such a threshold
indeed exists (Theorem~\ref{thm:gentheor}).

\subsection{Main results} \label{sec:intro:mainres}

The main contributions of this paper are fourfold. Firstly, we prove
(Theorem~\ref{thm:gentheor}) that for each $j \in [k-1]$, the
probability
\begin{equation} \label{eq:pj}
  \pj:= \frac{(j+1)\log n + \log\log n}{(k-j+1) n^{k-j}}(k-j)!
\end{equation}
is a sharp threshold for \formalconnectedness. Secondly, we prove a
hitting time result (also Theorem~\ref{thm:gentheor}), relating the
\connectedness\ threshold to the disappearance of all copies of the
\emph{minimal obstruction} $\Mj$ (Definition~\ref{def:mj}). Thirdly,
our results directly imply an analogous hitting time result for
$\cY_p$ (Corollary~\ref{cor:hittingtimeYp}), which Kahle and
Pittel~\cite{KahlePittel16} proved for $k=2$. Lastly, we analyse the
critical window given by the threshold $\pj$, showing that inside the
window the dimension of the $j$-th cohomology group converges in
distribution to a Poisson random variable
(Theorem~\ref{thm:critwindow1}).

Proving that $\pj$ is indeed a (sharp) threshold turns out to be
considerably more challenging than might be expected, largely because
\formalconnectedness\ of $\cG_p$
is \emph{not} a monotone increasing property. In
particular, the subcritical case is much more involved than it would
be for a monotone property, where often a simple second moment
argument suffices. In order to circumvent the difficulties arising
from the non-monotonicity, we introduce auxiliary structures called
\emph{local obstacles} (Definition~\ref{def:localobstacle}), showing
that whp $\cG_p$ evolves in a monotone way regarding those
(Lemma~\ref{lem:localobstacle}). In the supercritical case we must
guarantee that whp there are no more obstructions to \connectedness.
In order to bound the number of potential ``large'' obstructions,
basic calculations are not sufficient and therefore we define a
suitable search process, which gives us more precise bounds on their
number (Lemma~\ref{lem:largesupport}).  
 
Before defining the minimal obstruction $\Mj$
(Definition~\ref{def:mj}), we introduce the following necessary
concepts.

\begin{definition}\label{def:flower}
  Given a $k$-simplex $K$ in a $k$-dimensional simplicial complex
  $\cG$, a collection $\mathcal{F} = \{P_0, \ldots, P_{k-j}\}$ of
  $j$-simplices forms a \emph{$j$-flower in $K$} (see
  Figure~\ref{fig:flower}) if $K=\bigcup_{i=0}^{k-j} P_i$ and
  $C:=\bigcap_{i=0}^{k-j} P_i$ satisfies $|C|=j$. We call the
  $j$-simplices $P_i$ the \emph{petals} and the set $C$ the
  \emph{centre} of the $j$-flower $\mathcal{F}$.

  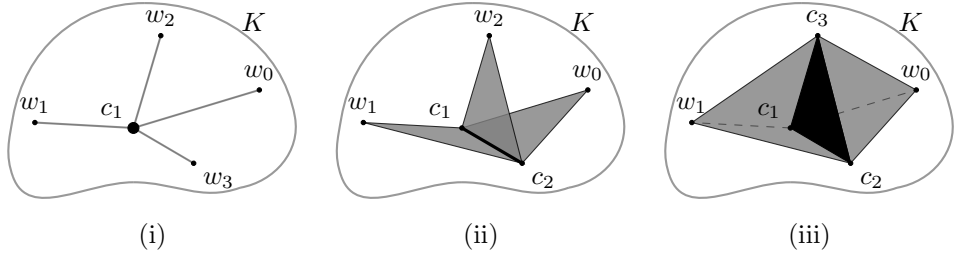
\begin{figure}[htbp] 
    \centering
    \begin{tikzpicture}[scale=0.72]
      \draw[thick] [opacity=0.4, fill=gray] (-7.6,1.7) -- (-6.5,1.05);
      \draw[thick] [opacity=0.4, fill=gray] (-7.6,1.7) -- (-7.1,3.4);
      \draw[thick] [opacity=0.4, fill=gray] (-7.6,1.7) -- (-9.4,1.8);
      \draw[thick] [opacity=0.4, fill=gray] (-7.6,1.7) -- (-5.3,2.4);
                
      \draw[fill] (-7.6,1.7) circle [radius=0.1] node [above left]    
        {$c_1$};
      \draw[fill] (-6.5,1.05)  circle [radius=0.04] node [below right]
        {$w_3$};
      \draw[fill] (-7.1,3.4) circle [radius=0.04] node [above]
        {$w_2$}; 
      \draw[fill] (-5.3,2.4) circle [radius=0.04]  node [above]
        {$w_0$}; 
      \draw[fill] (-9.4,1.8) circle [radius=0.04] node [above]
        {$w_1$};

      \draw[thick,opacity=0.4] (-9.8,1.9) to [out=80,in=180] (-7.4,4);
      \draw[thick,opacity=0.4] (-7.4,4) to [out=0,in=100] (-4.65,1.9);
      \draw[thick,opacity=0.4] (-4.65,1.9) to [out=280,in=10]
        (-5.65,0.6);
      \draw[thick,opacity=0.4] (-5.65,0.6) to [out=200,in=0]
        (-7.6,0.7);
      \draw[thick,opacity=0.4] (-7.6,0.7) to [out=180,in=290]
        (-9.8,0.8);
      \draw[thick,opacity=0.4] (-9.8,0.8) to [out=110,in=260]
        (-9.8,1.9);
   
      \node at (-5.4,3.65) {$K$};

      \node at (-7.25,-0.3) {(i)};

      \filldraw [opacity=0.8, fill=gray] (-1.6,1.7) -- (-0.5,1.05) --
        (-3.4,1.8) -- cycle; 
      \filldraw [opacity=0.8, fill=gray] (-1.6,1.7) -- (-0.5,1.05) --
        (0.7,2.4) -- cycle; 
      \filldraw [opacity=0.8, fill=gray] (-1.6,1.7) -- (-0.5,1.05) --
        (-1.1,3.4) -- cycle; 

      \draw[very thick] (-1.6,1.7) -- (-0.5,1.05);  
                
      \draw[fill] (-1.6,1.7) circle [radius=0.04] node [above left]
        {$c_1$};
      \draw[fill] (-0.5,1.05)  circle [radius=0.04] node [below right]
        {$c_2$};
      \draw[fill] (-1.1,3.4) circle [radius=0.04] node [above]
        {$w_2$}; 
      \draw[fill] (0.7,2.4) circle [radius=0.04]  node [above]
        {$w_0$}; 
      \draw[fill] (-3.4,1.8) circle [radius=0.04] node [above]
        {$w_1$}; 
       
      \draw[thick,opacity=0.4] (-3.8,1.9) to [out=80,in=180] (-1.4,4);
      \draw[thick,opacity=0.4] (-1.4,4) to [out=0,in=100] (1.35,1.9);
      \draw[thick,opacity=0.4] (1.35,1.9) to [out=280,in=10]
        (0.35,0.6);
      \draw[thick,opacity=0.4] (0.35,0.6) to [out=200,in=0]
        (-1.6,0.7);
      \draw[thick,opacity=0.4] (-1.6,0.7) to [out=180,in=290]
        (-3.8,0.8);
      \draw[thick,opacity=0.4] (-3.8,0.8) to [out=110,in=260]
        (-3.8,1.9);
   
      \node at (0.6,3.65) {$K$};

      \node at (-1.25,-0.3) {(ii)};

      \filldraw [opacity=0.8, fill=gray] (4.9,3.4) -- (5.5,1.05) --
        (2.6,1.8) -- cycle;
      \filldraw [opacity=0.8, fill=gray] (4.9,3.4) -- (5.5,1.05) --
        (6.7,2.4) -- cycle;
      \filldraw [fill=black] (4.4,1.7) -- (5.5,1.05) -- (4.9,3.4) --
        cycle; 
                
      \draw[fill] (4.4,1.7) circle [radius=0.04] node [above left]
        {$c_1$};
      \draw[fill] (5.5,1.05)  circle [radius=0.04] node [below right]
        {$c_2$};
      \draw[fill] (4.9,3.4) circle [radius=0.04] node [above] {$c_3$};
      \draw[fill] (6.7,2.4) circle [radius=0.04]  node [above]
        {$w_0$}; 
      \draw[fill] (2.6,1.8) circle [radius=0.04] node [above] {$w_1$};

      \draw[dashed,opacity=0.5] (4.4,1.7) -- (6.7,2.4);
      \draw[dashed,opacity=0.5] (4.4,1.7) -- (2.6,1.8);     
       
      \draw[thick,opacity=0.4] (2.2,1.9) to [out=80,in=180] (4.6,4);
      \draw[thick,opacity=0.4] (4.6,4) to [out=0,in=100] (7.35,1.9);
      \draw[thick,opacity=0.4] (7.35,1.9) to [out=280,in=10]
        (6.35,0.6);
      \draw[thick,opacity=0.4] (6.35,0.6) to [out=200,in=0] (4.4,0.7);
      \draw[thick,opacity=0.4] (4.4,0.7) to [out=180,in=290]
        (2.2,0.8);
      \draw[thick,opacity=0.4] (2.2,0.8) to [out=110,in=260]
        (2.2,1.9);
   
      \node at (6.6,3.65) {$K$};
    
      \node at (4.75,-0.3) {(iii)};    
    \end{tikzpicture}
    \caption{Examples of $j$-flowers in a $k$-simplex $K$, for $k=4$
      and $j=1,2,3$.\newline
      (i) The $1$-flower in $K$ with centre $C=\{c_1\}$ (bold black)
      and petals $P_i = C \cup \{w_i\}$, $i=0,1,2,3$ (grey).\newline
      (ii) The $2$-flower in $K$ with centre $C=\{c_1,c_2\}$ (bold
      black) and petals $P_i = C \cup \{w_i\}$, $i=0,1,2$
      (grey).\newline
      (iii) The $3$-flower in $K$ with centre $C=\{c_1,c_2,c_3\}$
      (bold black) and petals $P_i = C \cup \{w_i\}$, $i=0,1$ (grey).}
    \label{fig:flower}
  \end{figure}
        
  Observe that for each $k$-simplex $K$ and each $(j-1)$-simplex $C
  \subseteq K$, there is a unique $j$-flower in $K$ with centre $C$,
  namely
  \begin{equation}\label{eq:flowerinK}
    \mathcal{F}(K,C) := \{C \cup \{w\} \mid w \in K \setminus C\}.
  \end{equation}
\end{definition}
\noindent
When $j$ is clear from the context, we simply refer to a $j$-flower as
a \emph{flower}. 

A \emph{$j$-cycle} is a set $J$ of $j$-simplices such that every
$(j-1)$-simplex is contained in an \emph{even} number of $j$-simplices
in $J$.

\begin{definition}\label{def:mj}
  A \emph{copy of $\Mj$} (see Figure~\ref{fig:mj}) in a $k$-complex
  $\cG$ is a triple $(K,C,J)$ where
  \renewcommand{\theenumi}{(M\arabic{enumi})}
  \begin{enumerate}
  \item\label{Mj:simplex}
    $K$ is a $k$-simplex in $\cG$;
  \item\label{Mj:flower}
    $C$ is a $(j-1)$-simplex in $K$ such that each petal of the flower
    $\mathcal{F}=\mathcal{F}(K,C)$ is contained in \emph{no other}
    $k$-simplex of $\cG$;
  \item\label{Mj:cycle}
    $J$ is a $j$-cycle in $\cG$ that contains \emph{exactly one petal}
    of the flower $\mathcal{F}$, i.e.\ there exists a vertex $w_0 \in
    K \setminus C$ such that 
    \begin{equation*}
      J \cap \mathcal{F} =\Big\{ C \cup \{w_0\} \Big\}.
    \end{equation*}
  \end{enumerate}
  \renewcommand{\theenumi}{(\roman{enumi})}
\end{definition}

\begin{figure}[htbp]
  \centering
  \begin{tikzpicture}[scale=1.0]
    \draw[opacity=0.85, pattern=north west lines, pattern color=gray]
      (-1.5,2) -- (-0.6,1) -- (0.8,2.2) -- cycle;
    \filldraw [opacity=0.85,fill=gray] (-1.5,2) -- (-0.6,1) --
      (-3.1,1.0)  -- cycle;
    \filldraw [opacity=0.85,fill=gray] (-1.5,2) -- (-0.6,1) --
      (-2.9,2.2)  -- cycle;
    \filldraw [opacity=0.85, fill=gray] (-1.5,2) -- (-0.6,1) --
      (-1.2,3.5) -- cycle;
    
    \draw[very thick] (-0.6,1) -- (-1.5,2);  
        
    \draw[fill] (-1.5,2) circle [radius=0.04] node [above left]
      {$c_1$};
    \draw[fill] (-0.6,1)  circle [radius=0.04] node [below right]
      {$c_2$};
    \draw[fill] (0.8,2.2) circle [radius=0.04] node [above] {$w_0$};
    \draw[fill] (-1.2,3.5) circle [radius=0.04]  node [above] {$w_1$};
    \draw[fill] (-2.9,2.2) circle [radius=0.04] node [above] {$w_2$}; 
    \draw[fill] (-3.1,1.0) circle [radius=0.04]  node [left] {$w_3$}; 
      
    \draw[fill,opacity=0.85,pattern=north west lines,pattern color=gray]
      (2.7,2) -- (3.6,1) -- (5,2.2) -- cycle;
    \filldraw [opacity=0.25,fill=gray] (3.6,1) -- (5,2.2) -- (4.9,0.6)
      -- cycle;
    \filldraw [opacity=0.25,fill=gray] (3.6,1) -- (4.9,0.6) --
      (3.1,0.2) -- cycle;
    \filldraw [opacity=0.25,fill=gray] (3.6,1) -- (3.1,0.2) -- (2.7,2)
      -- cycle;

    \draw[dashed,opacity=0.5] (2.7,2) -- (4.9,0.6);
    
    \draw[fill] (2.7,2) circle [radius=0.04] node[above left] {$c_1$};
    \draw[fill] (3.6,1) circle [radius=0.04] node[below right]
      {$c_2$};
    \draw[fill] (5,2.2) circle [radius=0.04] node [above] {$w_0$};
    \draw[fill] (4.9,0.6) circle [radius=0.04] node [below] {$j_1$};
    \draw[fill] (3.1,0.2) circle [radius=0.04] node [below] {$j_2$};
    
    \node at (5.3,1.0) {$J$};

    \draw[thick,opacity=0.4] (-3.6,1.9) to [out=80,in=180] (-1.4,3.9);
    \draw[thick,opacity=0.4] (-1.4,3.9) to [out=0,in=100] (1.35,1.9);
    \draw[thick,opacity=0.4] (1.35,1.9) to [out=280,in=10] (0.35,0.6);
    \draw[thick,opacity=0.4] (0.35,0.6) to [out=200,in=0] (-1.6,0.7);
    \draw[thick,opacity=0.4] (-1.6,0.7) to [out=180,in=290]
      (-3.6,0.8);
    \draw[thick,opacity=0.4] (-3.6,0.8) to [out=110,in=260]
      (-3.6,1.9);
   
    \node at (0.45,3.65) {$K$};
   
    \node at (-1.2,-0.6) {(i)};
    \node at (4,-0.6) {(ii)};
  \end{tikzpicture}
  \caption{A copy of $M_j$, for $k=5$ and $j=2$.
    The striped $j$-simplices are identified.\newline
    (i) The $k$-simplex $K$ that contains the flower
    $\mathcal{F}(K,C)$ with centre $C = \{c_1,c_2\}$ and petals
    $P_i=C \cup \{w_i\}$, for $i=0,1,2,3$. Each petal $P_i$ is
    contained in no other $k$-simplex except $K$.\newline
    (ii) The $j$-cycle $J$ consisting of the $j$-simplices
    $P_0=\{c_1,c_2,w_0\}$, $\{c_2,w_0,j_1\}$,  $\{c_2,j_1,j_2\}$,
    $\{c_1,c_2,j_2\}$, $\{c_1,j_1,j_2\}$ and  $\{c_1,w_0,j_1\}$. It
    intersects the flower $\mathcal{F}(K,C)$ only in the petal
    $P_0$.}
  \label{fig:mj}
\end{figure}
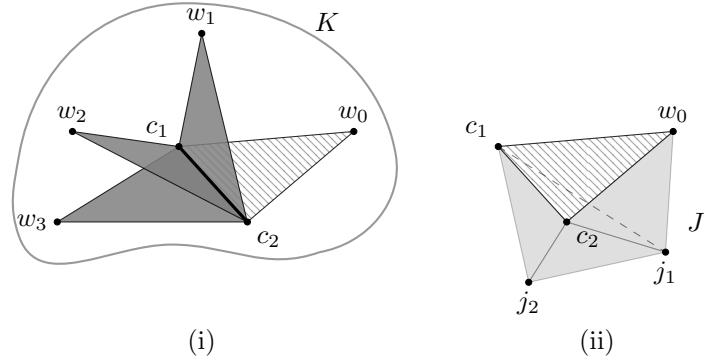
We will see in Section~\ref{sec:intuition:minob} that a copy of $M_j$
can be interpreted as a minimal obstruction for \formalconnectedness.

The random $k$-complex $\cG_p$ can be viewed as a \emph{process}, by
assigning a \emph{birth time} to each $k$-simplex. More precisely, for
each $(k+1)$-set of vertices in $[n]$ independently, sample a birth
time uniformly at random from $[0,1]$. (With probability $1$ no two
$(k+1)$-sets have the same birth time.)
Then $\cG_p$ is exactly the complex generated by the $(k+1)$-sets with
birth times at most $p$, by taking the downward-closure. If $p$ is
gradually increased from $0$ to $1$, we may interpret $\cG_p$ as a
process. Thus, we can define $\pMj$ as the birth time of the
$k$-simplex whose appearance causes the last copy of $\Mj$ to
disappear. More formally, let
\begin{equation} \label{eq:pMj}
  \pMj := \sup \{ p \in [0,1] \mid \cG_p \mbox{ contains a copy of }
  \Mj \}.
\end{equation}

Our first main result states that the value $\pMj$ is the hitting time
for \connectedness\ of $\cG_p$ and  is ``close'' to $\pj$ defined
in~\eqref{eq:pj}, implying that $\pj$ is in fact a sharp threshold for
\formalconnectedness.

\begin{thm} \label{thm:gentheor}
  Let $k\geq2$ be an integer and let $\omega$ be any function of $n$
  which tends to infinity as $n \rightarrow \infty$. For each $j \in
  [k-1]$, with high probability the following statements hold.
  \begin{enumerate}
  \item\label{thm:gentheor:pMj}
    $\frac{(j+1)\log n+\log\log n-\omega}{(k-j+1)n^{k-j}}(k-j)! < \pMj
    < \frac{(j+1)\log n+\log\log n+\omega}{(k-j+1)n^{k-j}}(k-j)!$.
  \item\label{thm:gentheor:subcrit}
    For all $p<\pMj$, $\cG_p$ is not \formalconnected, i.e.
    \begin{equation*}
      H^0(\cG_p;\FF_2)\neq \FF_2 \quad \mbox{ or } \quad H^{i}(\cG_p;\FF_2)\neq 0
      \mbox{ for some }i \in [j].
    \end{equation*}
  \item\label{thm:gentheor:supercrit}
    For all $p\geq\pMj$, $\cG_p$ is \formalconnected, i.e.
    \begin{equation*}
      H^0(\cG_p;\FF_2)= \FF_2 \quad \mbox{ and } \quad H^{i}(\cG_p;\FF_2) = 0
      \mbox{ for all } i \in [j].
    \end{equation*}
  \end{enumerate}
\end{thm} 

For the case $j=k-1$, Theorem~\ref{thm:gentheor} gives a threshold
$p_{k-1}=\frac{k \log n + \log\log n}{2n}$ for \mwconnectedness, which
is about half as large as the threshold $\frac{k \log n}{n}$ in
Theorem~\ref{thm:meshwall} for $\cY_p$. The reason for this is that
the minimal obstructions are different:\ in $\cY_p$ the minimal
obstruction is a $(k-1)$-simplex which is not contained in any
$k$-simplex of the complex (such a $(k-1)$-simplex is called
\emph{isolated}). By definition, isolated $(k-1)$-simplices do not
exist in $\cG_p$, because $\cG_p$ contains only those
$(k-1)$-simplices that lie in some $k$-simplex. 

Observe that Theorem~\ref{thm:gentheor}~\ref{thm:gentheor:subcrit}
and~\ref{thm:gentheor:supercrit} provide a hitting time result for the
process described above. A similar result was proved by Kahle and
Pittel~\cite{KahlePittel16} for $\cY_p$, but only for the
two-dimensional case. They considered the random complex process
associated with $\cY_p$ and related the vanishing of the first
cohomology group to the disappearance of the last isolated edge (i.e.\
$1$-simplex). As a corollary of Theorem~\ref{thm:gentheor}, we obtain
a hitting time result for $\cY_p$ for general $k\geq 2$. To
this end, let 
\begin{equation*}
  \pisol := \sup\{p\in[0,1] \mid \cY_p\mbox{ contains isolated
  $(k-1)$-simplices}\}
\end{equation*}
be the birth time of the $k$-simplex whose appearance causes the last
isolated $(k-1)$-simplex in $\cY_p$ to disappear and let
\begin{equation*}
  \pconn := \sup\{p\in[0,1] \mid H^{k-1}(\cY_p;\FF_2)\neq0\}
\end{equation*}
be the time when $\cY_p$ becomes \mwconnected.

\begin{cor}\label{cor:hittingtimeYp}
  Let $k\geq 2$ be an integer. Then, with high probability
  \begin{equation*}
    \pconn = \pisol.
  \end{equation*}
  In other words, with high probability the random process associated
  with $\cY_p$ becomes \mwconnected\ at exactly the moment when the last 
  isolated $(k-1)$-simplex disappears.
\end{cor}

Our last main result gives an explicit expression for the limiting
probability of the random complex $\cG_p$ being \formalconnected\
inside the critical window given by the threshold $\pj$ (cf.\
Theorems~\ref{thm:ERcritical} and~\ref{thm:KPcritical}). More
generally, we prove that the \emph{dimension} of the $j$-th cohomology
group with coefficients in $\FF_2$ converges in distribution to a
Poisson random variable.

\begin{thm}\label{thm:critwindow1}
  Let $k\geq 2$ be an integer, $j\in[k-1]$ and $\const\in\mathbb{R}$
  be a constant. Suppose that $(\const_n)_{n \geq 1}$ is a sequence of real
  numbers that converges to $\const$ as $n \rightarrow \infty$. If
  \begin{equation*}
    p = \frac{(j+1)\log n+\log\log n+\const_n}{(k-j+1)n^{k-j}}(k-j)!,
  \end{equation*}
  then $\dim \left( H^j(\cG_p; \FF_2) \right)$ converges in
  distribution to a Poisson random variable with expectation 
  \begin{equation*}
    \lambda_j := \frac{(j+1)e^{-\const}}{(k-j+1)^2 j!},
  \end{equation*} 
  while whp $H^0(\cG_p;\FF_2)=\FF_2$ and $H^i(\cG_p;\FF_2)=0$ for all
  $i \in [j-1]$. In particular,
  \begin{equation*}
    \Pr \left( \cG_p \mbox{ is \connected} \right)
    \xrightarrow{n \rightarrow \infty } e^{-\lambda_j}.
  \end{equation*}
\end{thm}
\noindent
Indeed, in the proof we will see that whp
$\dim\left(H^j(\cG_p;\FF_2)\right)$ equals the number of pairs $(K,C)$
for which there exists a $j$-cycle $J$ such that $(K,C,J)$ is a copy
of $\Mj$ in $\cG_p$.

\subsection{Related work}

This paper draws inspiration from~\cite{LinialMeshulam06} 
and~\cite{MeshulamWallach08}, but the proof techniques are
considerably different. We first note that in $\cY_p$ the presence of
the full $(k-1)$-dimensional skeleton trivially yields the topological
connectedness of $\cY_p$ and the vanishing of all the $i$-th
cohomology groups with $i \in [k-2]$. This is not true in $\cG_p$ and
therefore we need to consider \emph{all} cohomology groups up to
dimension $j$, for each $j\in[k-1]$. 

Moreover, in~\cite{LinialMeshulam06} and~\cite{MeshulamWallach08} one
standard application of the second moment method is sufficient for the
analysis of the subcritical case (i.e.\ statement (i)) of
Theorem~\ref{thm:meshwall}. By contrast, \formalconnectedness\ 
of $\cG_p$ is \emph{not} a monotone increasing property
(see Example~\ref{ex:nonmono}). This makes the subcritical case far
from trivial. More precisely, it does not suffice to prove that
$\cG_p$ is not \connected\ at
$p_-=\frac{(j+1)\log n+\log\log n-\omega}{(k-j+1)n^{k-j}}(k-j)!$;
rather we need to show that whp the property is not satisfied for
\emph{any} $p$ up to and including $p_-$. Also observe that in terms
of our hitting time result, it is \emph{not} enough to show that for
each ``small'' $p$ whp $\cG_p$ is not \connected. Rather, we need to
know that $\cG_p$ is not \connected\ whp \emph{for all such $p$
simultaneously}. 

The proof of the supercritical case $p\geq \pMj$ is also more
challenging than for $\cY_p$; we are forced to derive stronger bounds
for the number of \emph{bad functions} (see
Definition~\ref{def:badfctn}), due to the fact that for $j=k-1$, the
threshold in Theorem~\ref{thm:gentheor} is about half as large as the
corresponding threshold in~\cite{MeshulamWallach08}. To this end, we
define a breadth-first search process that makes use of the new notion
of \emph{traversability} (Definition~\ref{def:traversability}).
Moreover, non-monotonicity of \connectedness\ forces us to prove that
for all $p\geq \pMj$, the probability of $\cG_p$ \emph{not} being
\connected\ is small enough that we can apply a union bound over all
relevant values of $p$.

\subsection{Paper overview}

This paper is structured as follows.

In Section~\ref{sec:preliminaries} we present some preliminary results
that we will use throughout the paper and we provide an overview of
cohomology theory, which will allow us to define the concept of a
\emph{bad function} (see Definition~\ref{def:badfctn}), a
configuration in a complex $\cG$ that is a witness for
$H^j(\cG;\FF_2)$ not vanishing. Section~\ref{sec:intuition} is devoted
to the main concepts and the proof ideas used in this paper. After
explaining why a copy of $\Mj$ is a minimal obstruction to
\connectedness, we heuristically show why the value $p_j$ defined
in~\eqref{eq:pj} should be the threshold for \connectedness\ and give
an outline of the proofs of our main theorems. 

In Section~\ref{sec:subcrit}, we provide auxiliary results needed for
the proofs of Theorem~\ref{thm:gentheor}~\ref{thm:gentheor:pMj}
and~\ref{thm:gentheor:subcrit}. We analyse the subcritical case when
$p<\pMj$ and determine the approximate value of $\pMj$, i.e.\ when the
last minimal obstruction disappears. In Section~\ref{sec:supcrit} we
define a breadth-first search process which will allow us to examine
the supercritical case when $p\geq \pMj$ and to obtain results
necessary for the proofs of
Theorem~\ref{thm:gentheor}~\ref{thm:gentheor:supercrit} and
Theorem~\ref{thm:critwindow1}. 

We prove the main results Theorems~\ref{thm:gentheor}
and~\ref{thm:critwindow1} and Corollary~\ref{cor:hittingtimeYp} in
Section~\ref{sec:proofs}, using the auxiliary results from
Sections~\ref{sec:subcrit} and~\ref{sec:supcrit}. Finally, in
Section~\ref{sec:concremarks}, we discuss some open problems.

\section{Preliminaries} \label{sec:preliminaries}

\subsection{Birth times} \label{sec:preliminaries:birth}

We mentioned in Section~\ref{sec:intro:mainres} how to use the
standard birth times interpretation to describe the binomial model
$\cG_p$ as a process. In this setting, it is useful to introduce the
operation of ``adding a simplex''.

\begin{definition} \label{def:addingsimplex}
  Given a complex $\cG$ on vertex set $V$ and a non-empty set $B
  \subseteq V$, we define $\cG + B$ to be the complex obtained by
  adding the set $B$ and its downward-closure to $\cG$, i.e.\
  \begin{equation*}
    \cG + B := \cG \cup \{ 2^B \setminus \emptyset \}.
  \end{equation*}
\end{definition} 
\noindent
Observe that if $B$ is already a simplex of $\cG$, then $\cG + B =
\cG$. With this operation, $\cG_p$ (interpreted as a process) may also
be described in the following way. If $p_K$ is the smallest birth time
larger than $p$ of any $k$-simplex $K$, then $\cG_{p_K}=\cG_p + K$.

A property $\mathcal{P}$ of $k$-complexes is called \emph{monotone
increasing} if $\mathcal{P}$ is closed under adding $k$-simplices. The
complement of a monotone increasing property is called \emph{monotone
decreasing}. Finally, $\mathcal{P}$ is \emph{monotone} if it is
monotone increasing or decreasing.

Considering the birth times interpretation, we shall take union bounds
over finite sets of birth times. With a slight abuse of terminology,
sometimes we will talk about taking ``union bounds over $p$'' in some
interval, which makes little sense if we think of $p$ as being able to
take any value within the interval, but indeed we are conditioning on
the set of birth times and taking the union bound over all birth times
in the relevant interval.

We also note that conditioned on a $k$-simplex not being present at
time $p=q_1$, the probability that it is present at time $q_2$ is
$\frac{q_2-q_1}{1-q_1}$. Thus we may obtain $\cG_{q_2}$ from
$\cG_{q_1}$ by \emph{exposing} an additional probability of
$\frac{q_2-q_1}{1-q_1}$. Since we will only ever want to consider such
a situation with $q_1=o(1)$, we often simply take $q_2-q_1$ as an
approximation (and lower bound) for $\frac{q_2-q_1}{1-q_1}$, or use
$q_2$ as an upper bound.

\subsection{Probabilistic tools} \label{sec:preliminaries:prob}

We frequently use the following Chernoff bound.

\begin{lem}[see e.g.~{\cite[Theorem 2.1]{JansonLuczakRucinskiBook}}]
  \label{lem:chernoff}
  Given a binomial random variable $X$ with expectation $\mu$ and a
  real number $a>0$,
  \begin{align*}
    \Pr(X\ge \mu + a) &\le \exp\left(-\frac{a^2}{2(\mu+a/3)}\right);\\
    \Pr(X\le \mu - a) &\le \exp \left( - \frac{a^2}{2\mu}\right).
  \end{align*}
\end{lem}

For the analysis of the critical window (cf.\
Theorem~\ref{thm:critwindow1}), we will need the method of moments, as
presented in the following lemma.

\begin{lem}[see e.g.~{\cite[Theorem~20.11]{FriezeKaronski16}}]
  \label{lem:metmoments}
  Let $(S_n)_{n \geq 1}$ be a sequence of sums of indicator random
  variables. Suppose that there exists $\lambda>0$ such that for every
  fixed integer $t\geq 1$
  \begin{equation*}
    \lim_{n\to \infty} \EE \binom{S_n}{t} = \frac{\lambda^t}{t!}.
  \end{equation*}
  Then, for every integer $s\geq 0$,
  \begin{equation*}
    \lim_{n\to\infty}\Pr(S_n = s) = e^{-\lambda}\frac{\lambda^s}{s!},
  \end{equation*}
  i.e.\ $S_n$ converges in distribution to a Poisson random variable
  with expectation $\lambda$. We write $S_n \xrightarrow{d}
  \mbox{Po}(\lambda)$.
\end{lem}

\subsection{Cohomology terminology} \label{sec:preliminaries:cohomology}

We formally introduce cohomology groups with coefficients in $\FF_2$
for a simplicial complex. The following notions are all standard,
except the definition of a bad function
(Definition~\ref{def:badfctn}).
   
Given a $k$-complex $\cG$, for each $j \in \{0,\ldots,k\}$
denote by $C^j(\cG)$ the set of \emph{$j$-cochains}, that is, the set
of 0-1 functions on the $j$-simplices. The \emph{support} of a
function in $C^j(\cG)$ is the set of $j$-simplices mapped to 1. Each
$C^j(\cG)$ forms a group with respect to point-wise addition modulo 2.
We define the \emph{coboundary operators} $\delta^j\colon C^j(\cG)\to
C^{j+1}(\cG)$ for $j=0,\ldots,k-1$ as follows: for $f\in C^j(\cG)$,
the $(j+1)$-cochain $\delta^jf$ assigns to each $(j+1)$-simplex
$\sigma$ the value 
\begin{equation*}
  \delta^jf(\sigma) :=
  \sum_{\tau\subset\sigma,\mbox{ }|\tau|= j+1}{f(\tau)} \quad 
  (\bmod \mbox{ } 2).
\end{equation*}

In addition, we denote by $\delta^{-1}$ the unique group homomorphism
$\delta^{-1}\colon \{0\} \rightarrow C^0(\cG)$. The $j$-cochains in
$\im\delta^{j-1}$ are called \emph{$j$-coboundaries}, and the
$j$-cochains in $\ker\delta^{j}$ are called \emph{$j$-cocycles}. A
straightforward calculation shows that each coboundary operator is a
group homomorphism and that every $j$-coboundary is also a
$j$-cocycle, i.e.\ $\im\delta^{j-1}$ is a subgroup of $\ker\delta^{j}$.
Therefore, we can define the $j$-th cohomology group of $\cG$ with
coefficients in $\FF_2$ as the quotient group
\begin{equation} \label{eq:cohomgroup}
  H^j(\cG;\FF_2) := \ker\delta^j / \im\delta^{j-1}.
\end{equation}

By definition, $H^j(\cG;\FF_2)$ vanishes if and only if every
$j$-cocycle is a $j$-coboundary. This motivates the following
definition of a \emph{bad function}.

\begin{definition} \label{def:badfctn}
  For a $k$-complex $\cG$ and $j \in [k-1]$,
  we say that a function $f\in C^j(\cG)$ is \emph{bad} if
  \begin{enumerate}
  \item $f$ is a $j$-cocycle, i.e.\ it assigns an even number of 1's
    to the $j$-simplices on the boundary of each $(j+1)$-simplex;
  \item $f$ is not a $j$-coboundary, i.e.\ it is not induced by a
    $(j-1)$-cochain.
  \end{enumerate}
\end{definition}
\noindent
Thus, $H^j(\cG;\FF_2)$ vanishes if and only if no bad function in
$C^j(\cG)$ exists. 

Recall that a set of $j$-simplices is a $j$-cycle if every
$(j-1)$-simplex is contained in an \emph{even} number of
$j$-simplices of the set. It is easy to see that if $f$ is a
$j$-cocycle and $J$ is a $j$-cycle such that the restriction $f|_J$
has support of odd size, then $f$ is not a $j$-coboundary and thus is
a bad function.

\section{Intuition and outline of proofs} \label{sec:intuition}

\noindent For the rest of the paper, let $j\in [k-1]$ be fixed.

\subsection{Minimal obstructions} \label{sec:intuition:minob}

Let us explain why  $\Mj$ (Definition~\ref{def:mj}) can be interpreted
as the (unique) minimal obstruction to \connectedness. Given a triple $(K,C,J)$
which forms a copy of $\Mj$ in a $k$-complex $\cG$, it is easy to
define a bad function $f\in C^j(\cG)$ (see
Definition~\ref{def:badfctn}): let $f$ take value $1$ on the petals of
the flower $\mathcal{F}(K,C)$ (see~\eqref{eq:flowerinK}) and $0$
everywhere else. Since the petals are all in the $k$-simplex $K$ but
in no further $k$-simplices, every $(j+1)$-simplex $L$ in $\cG$ is
even, because $L$ contains either two petals (if $C\subseteq
L\subseteq K$) or none (otherwise). However, $J$ would be a $j$-cycle
containing precisely one $j$-simplex, namely the petal $C\cup\{w_0\}$,
on which $f$ takes value $1$, ensuring that $f$ is not a
$j$-coboundary. Thus $f$ is bad and has support of size $k-j+1$, which
is the number of petals of $\mathcal{F}(K,C)$. 

In the following lemma we show that in fact such a bad function is the
only possibility for an obstruction which is minimal with respect to
the size of the support. Given a $k$-simplex $K$ and a collection $S$
of $j$-simplices, define $S_K$ to be the set of $j$-simplices of $S$
contained in $K$.

\begin{lem} \label{lem:minobst} 
  Let $\cG$ be a $k$-complex and let $S$ be the support of a
  $j$-cocycle. Then for each $k$-simplex $K$,
  \begin{enumerate}
  \item\label{minobst:size}
    either $S_K = \emptyset$ or both $|S_K|\geq k-j+1$ and
    $\bigcup_{\sigma \in S_K} \sigma = K$;
  \item\label{minobst:flower}
    if $|S_K|=k-j+1$, then $S_K$ forms a $j$-flower in $K$.
  \end{enumerate}
\end{lem} 

\begin{proof}
  \ref{minobst:size} Suppose $S_K \neq \emptyset$ and let $\sigma_0
  \in S_K$. Let the vertices of $K\setminus \sigma_0$ be denoted by
  $v_1, \ldots, v_{k-j}$.
  Each $(j+1)$-simplex $\sigma_0 \cup \{v_i\}$ has to be even with
  respect to $f$ and thus contains some $j$-simplex $\sigma_i \in S_K
  \setminus \{ \sigma_0\}$, which therefore contains $v_i$. The simplices
  $\sigma_0, \ldots, \sigma_{k-j}$ are distinct, because each $v_i$
  lies in $\sigma_i$ but in no other $\sigma_{i'}$. Therefore  $|S_K|
  \geq k-j+1$ and 
  
  \[ K \supseteq \bigcup_{\sigma \in S_K} \sigma \supseteq \sigma_0 \cup
  \{v_1,\ldots, v_{k-j}\} = K .\]

  \noindent
  \ref{minobst:flower} Suppose now that
  $S_K=\{\sigma_0,\ldots,\sigma_{k-j}\}$, with
  $\sigma_0,\dotsc,\sigma_{k-j}$ defined as above. For $2 \le i \le 
  k-j$, the $(j+1)$-simplex $\tau := \sigma_1\cup\{v_i\}$ contains
  $\sigma_1$, but no $\sigma_\ell$ with $\ell \notin \{1,i\}$. By the
  choice of $S$ as the support of a $j$-cocycle, $\tau$ is even and
  thus $\sigma_i \subset \tau$. This means that
  \begin{equation*}
    \sigma_1\cap\sigma_i = \tau\setminus\{v_1,v_i\} =
    \sigma_0\cap\sigma_1.
  \end{equation*}
  As this holds for all $i$, $S_K$ forms a flower in $K$ with centre
  $\sigma_0\cap\sigma_1$.
\end{proof}

Both the presence of a copy of $\Mj$ and \connectedness\ in $\cG_p$
are \emph{not} monotone properties, as the following example shows.

\begin{example}\label{ex:nonmono}
  Let $\cG$ be the $2$-complex on vertex set $\{1,2,3,4,5\}$ generated
  by the $3$-uniform hypergraph with hyperedges $\{1,2,3\}$ and
  $\{1,4,5\}$, see Figure~\ref{fig:nonmon}. Then $\cG$ is
  $1$-cohom-connected and thus contains no copies of $M_1$. Adding to
  $\cG$ the $2$-simplex $\{2,3,4\}$ (and its downward-closure) creates
  several copies of $M_1$ and thus yields a complex $\cG'$ which is
  not $1$-cohom-connected. If we further add the $2$-simplex
  $\{1,3,4\}$ to $\cG'$, we obtain a $2$-complex $\cG''$ which is
  $1$-cohom-connected and thus contains no copies of $M_1$.
\end{example}

\begin{figure}[htbp]
  \centering
  \begin{tikzpicture}[scale=0.75]
    \filldraw [opacity=0.5,fill=gray] (1.6,0) -- (0.8,2) -- (2.5,3.2)
      -- cycle;
    \filldraw [opacity=0.5,fill=gray] (3.4,0) -- (4.2,2) -- (2.5,3.2)
      -- cycle;
    
    \draw[fill] (2.5,3.2) circle [radius=0.05] node [above] {$1$};
    \draw[fill] (0.8,2) circle [radius=0.05] node [left] {$2$};
    \draw[fill] (1.6,0) circle [radius=0.05] node [below] {$3$};
    \draw[fill] (3.4,0) circle [radius=0.05] node [below] {$4$};
    \draw[fill] (4.2,2) circle [radius=0.05] node [right] {$5$};
  
    \node at (2.5,-0.8) {$\cG$};

    \begin{scope}[xshift=6cm]
      \filldraw [opacity=0.5,fill=gray] (1.6,0) -- (0.8,2) --
        (2.5,3.2) -- cycle;
      \filldraw [opacity=0.5,fill=gray] (3.4,0) -- (4.2,2) --
        (2.5,3.2) -- cycle;
      \filldraw [opacity=0.5,fill=gray] (1.6,0) -- (0.8,2) -- (3.4,0)
        -- cycle;
    
      \draw[fill] (2.5,3.2) circle [radius=0.05] node [above] {$1$};
      \draw[fill] (0.8,2) circle [radius=0.05] node [left] {$2$};
      \draw[fill] (1.6,0) circle [radius=0.05] node [below] {$3$};
      \draw[fill] (3.4,0) circle [radius=0.05] node [below] {$4$};
      \draw[fill](4.2,2) circle [radius=0.05] node [right] {$5$};
    
      \node at (2.5,-0.8) {$\cG'$};
    \end{scope}

    \begin{scope}[xshift=12cm]
      \filldraw [opacity=0.5,fill=gray] (1.6,0)  -- (0.8,2) --
        (2.5,3.2) -- cycle;
      \filldraw [opacity=0.5,fill=gray] (3.4,0)  -- (4.2,2)  --
        (2.5,3.2) -- cycle;
      \filldraw [opacity=0.5,fill=gray] (1.6,0) -- (0.8,2) -- (3.4,0)
        -- cycle;
      \filldraw [opacity=0.5,fill=gray] (1.6,0) -- (3.4,0) --
        (2.5,3.2) -- cycle;
    
      \draw[fill] (2.5,3.2) circle [radius=0.05] node [above] {$1$};
      \draw[fill] (0.8,2) circle [radius=0.05] node [left] {$2$};
      \draw[fill] (1.6,0) circle [radius=0.05] node [below] {$3$};
      \draw[fill] (3.4,0) circle [radius=0.05] node [below] {$4$};
      \draw[fill] (4.2,2) circle [radius=0.05] node [right] {$5$};
    
      \node at (2.5,-0.8) {$\cG''$};
    \end{scope}
  \end{tikzpicture}
  \caption{Adding simplices might create new copies of $\Mj$ or
    destroy existing ones.}
  \label{fig:nonmon}
\end{figure}
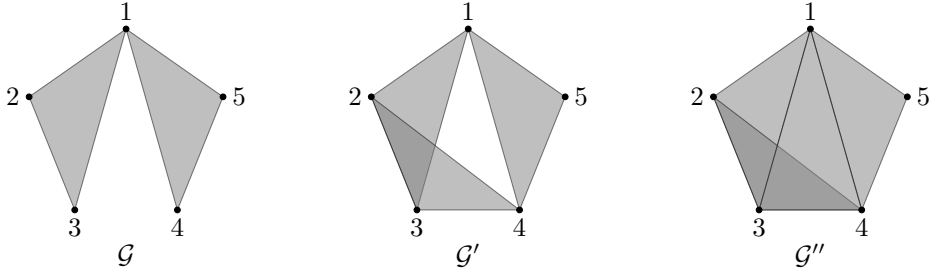

\subsection{Finding the threshold} \label{sec:intuition:heur}

In this section we provide a heuristic argument for why the threshold
for the disappearance of the last copy of $\Mj$ should be around
$\pj$. To do this, we will make use of a simplified version of the
obstruction $\Mj$.

\begin{definition}\label{def:mjminus}
  A \emph{copy of $\Mjmin$} (see Figure~\ref{fig:Mjmin}) in a
  $k$-complex $\cG$ is a pair $(K,C)$ where
  \renewcommand{\theenumi}{(M\arabic{enumi})}
  \begin{enumerate}
  \item\label{Mjmin:simplex}
    $K$ is a $k$-simplex in $\cG$;
  \item\label{Mjmin:flower}
    $C$ is a $(j-1)$-simplex in $K$ such that each petal of the flower
    $\mathcal{F}(K,C)$ is contained in \emph{no other}
    $k$-simplex of $\cG$.
  \end{enumerate}
  \renewcommand{\theenumi}{(\roman{enumi})}
\end{definition}

\begin{figure}[htbp]
  \centering
  \begin{tikzpicture}[scale=1]
    \draw[opacity=0.85, fill=gray] (-1.5,2) -- (-0.6,1) -- (0.8,2.2)
      -- cycle;
    \filldraw [opacity=0.85,fill=gray] (-1.5,2) -- (-0.6,1) --
      (-3.1,1.0)  -- cycle;
    \filldraw [opacity=0.85,fill=gray] (-1.5,2) -- (-0.6,1) --
      (-2.9,2.2)  -- cycle;
    \filldraw [opacity=0.85, fill=gray] (-1.5,2) -- (-0.6,1) --
      (-1.2,3.5) -- cycle;
    
    \draw[very thick] (-0.6,1) -- (-1.5,2);  
        
    \draw[fill] (-1.5,2) circle [radius=0.04] node [above left]
      {$c_1$};
    \draw[fill] (-0.6,1)  circle [radius=0.04] node [below right] 
      {$c_2$};
    \draw[fill] (0.8,2.2) circle [radius=0.04] node [above] {$w_0$}; 
    \draw[fill] (-1.2,3.5) circle [radius=0.04]  node [above] {$w_1$};
    \draw[fill] (-2.9,2.2) circle [radius=0.04] node [above] {$w_2$}; 
    \draw[fill] (-3.1,1.0) circle [radius=0.04]  node [left] {$w_3$}; 
      
    \draw[thick,opacity=0.4] (-3.6,1.9) to [out=80,in=180] (-1.4,3.9);
    \draw[thick,opacity=0.4] (-1.4,3.9) to [out=0,in=100] (1.35,1.9);
    \draw[thick,opacity=0.4] (1.35,1.9) to [out=280,in=10] (0.35,0.6);
    \draw[thick,opacity=0.4] (0.35,0.6) to [out=200,in=0] (-1.6,0.7);
    \draw[thick,opacity=0.4] (-1.6,0.7) to [out=180,in=290]
      (-3.6,0.8);
    \draw[thick,opacity=0.4] (-3.6,0.8) to [out=110,in=260]
      (-3.6,1.9);
   
    \node at (0.45,3.65) {$K$};
  \end{tikzpicture}
  \caption{A copy of $M_j^-$, for $k=5$ and $j=2$. The $k$-simplex $K$
    contains the flower $\mathcal{F}(K,C)$ with centre $C =
    \{c_1,c_2\}$ and petals $P_i=C \cup \{w_i\}$, for $i=0,1,2,3$.
    Each petal $P_i$ is contained in no other $k$-simplex except $K$.}
  \label{fig:Mjmin}
\end{figure}
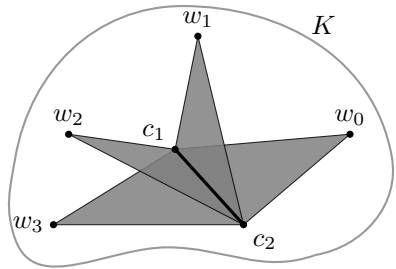
  
In other words, a copy of $\Mjmin$ can be viewed as a copy of $\Mj$
without the condition~\ref{Mj:cycle}, i.e.\ without the $j$-cycle $J$
containing one of the petals (see Figures~\ref{fig:mj}
and~\ref{fig:Mjmin}). Therefore,
\begin{equation*} 
  \Mjmin \not\subset \cG_p \hspace{3ex} \Rightarrow \hspace{3ex}
  \Mj\not\subset \cG_p.
\end{equation*}
Moreover, we will show (Lemma~\ref{lem:jshells}) that, for $p$
approaching the value $\pj$, the $j$-cycle $J$ needed to extend a copy
of $\Mjmin$ to a copy of $\Mj$ is very likely to exist. Hence in this
range the existence of $\Mjmin$ and $\Mj$ are essentially equivalent
events.

Let us estimate the expected number of copies of $\Mjmin$ in $\cG_p$.
The probability of $k+1$ arbitrary vertices with a fixed centre $C$
forming a copy of $\Mjmin$ is about $p(1-p)^{(k-j+1) \binom{n}{k-j}}$,
which we can approximate by 
\begin{equation*}
  pe^{-\frac{(k-j+1)n^{k-j}}{(k-j)!}p},
\end{equation*}
so the expected number of copies of $\Mjmin$ is of order
$n^{k+1}pe^{-\frac{(k-j+1)n^{k-j}}{(k-j)!}p}$. We seek $p$ such that
\begin{equation*}
  n^{k+1}pe^{-\frac{(k-j+1)n^{k-j}}{(k-j)!}p}=1.
\end{equation*}
This holds when
\begin{equation*}
  (k+1)\log n + \log p -\frac{(k-j+1)n^{k-j}}{(k-j)!}p = 0,
\end{equation*}
which implies
\begin{align*}
  p & = \frac{(k+1)\log n+\log p}{(k-j+1)n^{k-j}}(k-j)!\\
  & = \frac{(k+1)\log n+\log\left(\frac{(k+1)\log n+\log p}{k-j+1}(k-j)!\right) - (k-j) \log n}{(k-j+1)n^{k-j}} (k-j)!\\
  & = \frac{(j+1)\log n + \log\log n + O(1)}{(k-j+1) n^{k-j}}(k-j)!,
\end{align*}
which corresponds to the stated threshold $\pj$ defined
in~\eqref{eq:pj}.

\subsection{Outline of the proofs} \label{sec:intuition:outline}

We now give an outline of the proofs of our main theorems. Let us
begin with Theorem~\ref{thm:gentheor}. To analyse the zero-th
cohomology group, we define the probabilities
\begin{itemize}
\item $p_0:= \frac{\log n}{n^k}k!$;
\item $p_T:= \sup \{p\in[0,1] \mid \cG_p \mbox{ is not topologically
  connected}\}$. 
\end{itemize}
In other words, $p_T$ is the birth time of the $k$-simplex whose
appearance causes the complex $\cG_p$ to become topologically
connected. Recall that topological 
connectedness is equivalent to the random hypergraph
$G_p$ becoming vertex-connected. It is known (see
e.g.~\cite{CooleyKangKoch16,Poole14Hamilton,Poole15}) 
that $p_0$ is the threshold for
vertex-connectedness of the random $(k+1)$-uniform hypergraph, that is
whp $p_T = (1+o(1))p_0$ (Lemma~\ref{lem:topconn}). 

Recall from~\eqref{eq:pj} and~\eqref{eq:pMj} that for each $j\in[k-1]$
we have
\begin{itemize}
\item $\displaystyle \pj =
  \frac{(j+1)\log n + \log\log n}{(k-j+1) n^{k-j}}(k-j)!$;
\item $\pMj =
  \sup\{p\in[0,1] \mid \cG_p\mbox{ contains a copy of }\Mj\}$.
\end{itemize}
In other words, $\pMj$ is the birth time of the $k$-simplex whose
appearance causes the last copy of $M_{j}$ to disappear.

In Section~\ref{sec:subcrit}, we study the subcritical case when
$p<\pMj$, providing results needed for the proof of
Theorem~\ref{thm:gentheor}~\ref{thm:gentheor:subcrit}. Moreover, we
show that whp the value of $\pMj$ is ``close'' to $\pj$
(Corollary~\ref{cor:pMj}), thus proving
Theorem~\ref{thm:gentheor}~\ref{thm:gentheor:pMj}. 

In order to prove
Theorem~\ref{thm:gentheor}~\ref{thm:gentheor:subcrit}, we aim to show
that whp $H^j(\cG_p;\FF_2) \neq 0$ throughout the interval
$[\pMjmin,\pMj)$. A direct argument based on determining the
dimensions of $C^{j-1}(\cG_p)$, $C^{j}(\cG_p)$ and $C^{j+1}(\cG_p)$
may be considered, but it would work only for some values of $j$ and
some ranges of $p$ (see Section~\ref{sec:concremarks:dim}). We
actually prove a stronger result (Lemma~\ref{lem:phantomlemma}), for
which we define the following probabilities: for each $j\in[k-1]$, set
\begin{itemize}
\item $\displaystyle \pjmin :=
  \left(1-\frac{1}{\sqrt{\log n}}\right)\frac{(j+1)\log n}{(k-j+1)n^{k-j}}(k-j)!$;
\item $\displaystyle \pjone :=
  \frac{1}{10(j+1)\binom{k+1}{j+1}n^{k-j}}$.
\end{itemize}
We will also need the value 
\begin{equation*}
  p_0^- := \displaystyle \frac{\log n}{n^k} = \frac{p_0}{k!}.
\end{equation*} 
The motivation behind these seemingly arbitrary definitions will
become clear as the argument develops. We will prove that three copies
of $\Mj$ suffice to cover the interval $[\pjmom, \pMj)$, which whp
contains the interval $[\pMjmin,\pMj)$ by
Theorem~\ref{thm:gentheor}~\ref{thm:gentheor:pMj}.

\begin{lem}\label{lem:phantomlemma}
  Let $j\in[k-1]$. With high probability, there exist three triples
  $(K_\ell,C_\ell,J_\ell)$, $\ell=1,2,3$, such that for all
  $p\in[\pjmom,\pMj)$, $(K_\ell,C_\ell,J_\ell)$ forms a copy of $\Mj$
  in $\cG_p$ for some $\ell$. In particular, whp $H^j(\cG_p;\FF_2) \ne
  0$ for all $p \in [\pjmom,\pMj)$.
\end{lem}
\noindent
This will in particular imply that whp $\cG_p$ is not \connected\ in
the interval $[\pjmom,\pMj)$. 
By Lemma~\ref{lem:phantomlemma} applied
with $j$ replaced by $i$ for each $i \in [j]$ and by the fact that
$\cG_p$ is not topologically connected in $[0,p_T)$ by definition, whp
$\cG_p$ is not \connected\ in the range
\begin{equation*}
 [0,p_T) \cup \bigcup\limits_{i=1}^{j} [p_{i-1}^-,p_{M_i}) \quad
 \stackrel{\mbox{\footnotesize (whp)}}{=} \quad [0,\pMj).
\end{equation*}
This completely covers the subcritical case
(Theorem~\ref{thm:gentheor}~\ref{thm:gentheor:subcrit}). 

In order to prove Lemma~\ref{lem:phantomlemma}, we divide the interval
$[\pjmom,\pMj)$  into smaller subintervals 
\begin{equation*}
  [\pjmom,\pMj) = [\pjmom, \pjone] \cup [\pjone,\pjmin] \cup
  [\pjmin,\pMj)
\end{equation*}
and show that for each of these subintervals, whp there is one copy of
$\Mj$ which exists in $\cG_p$ throughout this interval, using the
following strategy.

\renewcommand{\theenumi}{(\Roman{enumi})}
\begin{enumerate}
\item At around $\pjmom$, whp there exist ``many'' copies of $\Mj$
  (Lemma~\ref{lem:secondmoment}) and whp at least one of these 
  survives until probability $\pjone$ (Lemma~\ref{lem:pj-1pjone}).
\item For any $p \geq \pjone$, whp all copies of $\Mjmin$ give rise
  to copies of $\Mj$, thus the existence of $\Mjmin$ and $\Mj$ are
  essentially equivalent events (Lemma~\ref{lem:jshells}). In
  particular, the last $\Mj$ to disappear corresponds to the last
  $\Mjmin$ (Corollary~\ref{cor:noMjmin}).
\item At around $\pjmin$, whp there are ``many'' copies of $\Mjmin$
  (Lemma~\ref{lem:meanconc}) and whp one of these already existed at
  $\pjone$ (Lemma~\ref{lem:pj1pjmin}).
\item The last $\Mjmin$ to disappear whp already existed at $\pjmin$
  (Lemma~\ref{lem:pjmin-pMj}). 
\end{enumerate}
\renewcommand{\theenumi}{(\roman{enumi})}

In Section~\ref{sec:supcrit} we study the supercritical case, i.e.\
the case $p \geq \pMj$, and derive auxiliary results, necessary to
prove Theorem~\ref{thm:gentheor}~\ref{thm:gentheor:supercrit}. By the
definition of $\pMj$, we know that $\cG_p$  contains no $\Mj$ in this
range, so by Lemma~\ref{lem:minobst} it remains to show that whp there
are no bad functions with support of size $s>k-j+1$. In other words,
we need to prove that each $j$-cocycle with support of size $s$ is
also a $j$-coboundary.

To this end, we prove (Corollaries~\ref{cor:noSpsmallp}~and~\ref{cor:supercritical}) that from
slightly before the threshold $\pj$ onwards, every $j$-cocycle can be written
as the sum of functions arising from copies of $\Mjmin$ (see
Definition~\ref{def:generating}). We first show
(Lemma~\ref{lem:traversable}) that the support of any smallest
$j$-cocycle \emph{not} generated by copies of $\Mjmin$ satisfies a property
which we call \emph{traversability}
(Definition~\ref{def:traversability}). We then bound the probability
that such a support of size $s$ exists. For constant $s$, simple
bounds will suffice (Lemma~\ref{lem:smallsupport}); for larger values
of $s$, traversability will allow us to define a breadth-first search
process that we use to track the construction of a traversable support
and thus count the number of such supports much more accurately
(Lemma~\ref{lem:largesupport}). 

Combining the results from Sections~\ref{sec:subcrit}
and~\ref{sec:supcrit}, we prove Theorem~\ref{thm:gentheor} in
Section~\ref{sec:proofs}. We then apply Theorem~\ref{thm:gentheor} to
derive Corollary~\ref{cor:hittingtimeYp}, which provides a hitting
time result for $\cY_p$, relating the vanishing of
$H^{k-1}(\cY_p;\FF_2)$ to the disappearance of the last isolated
$(k-1)$-simplex. 

Finally, we prove Theorem~\ref{thm:critwindow1} in
Section~\ref{sec:proofs:critwind}. We analyse \formalconnectedness\ of
$\cG_p$ within the critical window given by the threshold for this
property, i.e.\ we consider
$p=\frac{(j+1)\log n+\log\log n+O(1)}{(k-j+1)n^{k-j}}(k-j)!$. In this
range, whp all $j$-cocycles arise from copies of $\Mjmin$
(Corollary~\ref{cor:noSpsmallp}).
Using the method of moments
(Lemma~\ref{lem:metmoments}), we will show that the number of
copies of $\Mjmin$ converges in distribution to a Poisson random
variable and that whp this number equals the dimension of the $j$-th
cohomology group of $\cG_p$. Thus, in particular we derive an explicit
expression for the limiting probability of $\cG_p$ being \connected.

\section{Subcritical regime} \label{sec:subcrit}

In this section we study the subcritical case $p<\pMj$ and derive the
necessary results for the proofs of statements~\ref{thm:gentheor:pMj}
and~\ref{thm:gentheor:subcrit} of Theorem~\ref{thm:gentheor}.

\subsection{Topological connectedness}\label{sec:topconn}

We begin with a result stating that 
\begin{equation*}
  p_0 = \frac{\log n}{n^k}k!
\end{equation*}
is a sharp threshold for topological connectedness of $\cG_p$.
Recall that $p_T$ is the birth time of the $k$-simplex whose
appearance causes the complex $\cG_p$ to become topologically
connected.

\begin{lem}\label{lem:topconn}
  Let $\omega$ be any function of $n$ which tends to infinity as
  $n\to\infty$. Then with high probability
  \begin{equation*}
    \frac{\log n - \omega}{n^k}k! < p_T <\frac{\log n + \omega}{n^k}k!
  \end{equation*}
  and thus in particular $p_T > p_0^-$.
\end{lem}

Observe that Lemma~\ref{lem:topconn} is equivalent to $p_0$ being a
sharp threshold for vertex-connectedness of the random $(k+1)$-uniform
hypergraph, which follows for instance from~\cite{CooleyKangKoch16}
or~\cite{Poole15} as a special case of each (see
also~\cite{Poole14Hamilton} for a stronger result). The proof relies
on standard applications of the first and second moment methods and is
an easy generalisation of the graph case (see
e.g.~\cite{KaronskiLuczak96}).

\subsection{Counting obstructions}

In this section we provide several results concerning the number of
minimal obstructions that exist in $\cG_p$ whp. First we define a
special case of $\Mj$ (Definition~\ref{def:mjstar}), which will be
useful in the subsequent arguments.

\begin{definition} \label{def:shell}
  For any $(j+2)$-set A in a complex $\cG$, the collection of all
  $(j+1)$-subsets of $A$ is called a \emph{j-shell} if each of them
  forms a $j$-simplex in $\cG$. The $j$-shell is called \emph{hollow}
  if $A$ does not form a $(j+1)$-simplex in $\cG$. 
\end{definition}
If the collection of all $(j+1)$-subsets of a $(j+2)$-set $A$ forms a
$j$-shell, with a slight abuse of terminology we also refer to the set
$A$ itself as a $j$-shell.

\begin{definition} \label{def:mjstar}
  Given a $k$-complex $\cG$ on vertex set $[n]$, a $(k+1)$-set $K$ in
  $\cG$, a $j$-set $C \subseteq K$, and two vertices $w\in K\setminus
  C$ and $a\in[n]\setminus K$, we say that the $4$-tuple $(K,C,w,a)$
  forms a \emph{copy of $\Mjst$} (see Figure~\ref{fig:mjstar}) if
  \renewcommand{\theenumi}{(M\arabic{enumi})}
  \begin{enumerate}    
  \item\label{Mjst:simplex}
    $K$ is a $k$-simplex in $\cG$;
  \item\label{Mjst:flower}
   $C$ is a $(j-1)$-simplex in $K$ such that each petal of the flower
    $\mathcal{F}(K,C)$ is contained in \emph{no other}
    $k$-simplex of $\cG$;
    \renewcommand{\theenumi}{(M\arabic{enumi}*)}
  \item\label{Mjst:shell}
    $C \cup \{w\} \cup \{a\}$ is a $j$-shell in $\cG$.
  \end{enumerate}
  \renewcommand{\theenumi}{(\roman{enumi})}
  Recall that~\ref{Mjst:simplex} and~\ref{Mjst:flower} mean that
  $(K,C)$ forms a copy of $\Mjmin$ (see Definition~\ref{def:mjminus}).
  We call the $j$-simplex $C \cup \{w\}$ the \emph{base} and $a$ the
  \emph{apex vertex} of the $j$-shell $C\cup \{w\} \cup \{a\}$. Every
  other $j$-simplex in $C\cup\{w\}\cup\{a\}$ is called a \emph{side}
  of the $j$-shell.
\end{definition}

\begin{figure}[htbp]
\centering
  \begin{tikzpicture}[scale=1.0]
    \filldraw [opacity=0.85,fill=gray] (1.8,2.7) -- (4.4,1.4) --
      (2.8,1.8) -- cycle;
    \filldraw [opacity=0.85,fill=gray] (1.8,2.7) -- (4.1,3) --
      (2.8,1.8) -- cycle;
    \filldraw [opacity=0.85,fill=gray] (1.8,2.7) -- (2.5,4.3) --
      (2.8,1.8) -- cycle;
    \draw[fill,opacity=0.85,pattern=north west lines,pattern color=gray]
      (1.8,2.7) -- (0.6,2) node [above] {$w$} -- (2.8,1.8) -- cycle;

    \filldraw [opacity=0.25,fill=gray] (0.6,2) -- (1.5,0.3) --
      (2.8,1.8) -- cycle;
    \draw[dashed,opacity=0.5] (1.5,0.3)  -- (1.8,2.7);

    \draw[very thick] (1.8,2.7) -- (2.8,1.8);
    \draw[fill] (1.8,2.7) circle [radius=0.04] node [above left]
      {$c_1$};
    \draw[fill] (2.8,1.8) circle [radius=0.04] node [below right]
      {$c_2$};

    \draw[fill] (4.4,1.4) circle [radius=0.04];
    \draw[fill] (4.1,3) circle [radius=0.04];
    \draw[fill] (2.5,4.3) circle [radius=0.04];
    \draw[fill] (0.6,2) circle [radius=0.04];
    \draw[fill] (1.5,0.3) circle [radius=0.04]  node [below] {$a$};

    \draw[thick,opacity=0.4] (0.3,3.0) to [out=80,in=180] (2.5,4.8);
    \draw[thick,opacity=0.4] (2.5,4.8) to [out=0,in=100] (5.2,2.5);
    \draw[thick,opacity=0.4] (5.2,2.5) to [out=280,in=20] (4.25,1.15);
    \draw[thick,opacity=0.4] (4.25,1.15) to [out=190,in=0] (2.0,1.5);
    \draw[thick,opacity=0.4] (2.0,1.5) to [out=180,in=290] (0.35,2.0);
    \draw[thick,opacity=0.4] (0.35,2.0) to [out=110,in=260] (0.3,3.0);
    \node at (4.2,4.5) {$K$};
  \end{tikzpicture}
  \caption{A copy of $M_j^*$, for $k=5$ and $j=2$. The pair $(K,C)$,
    with $K$ a $k$-simplex and $C=\{c_1,c_2\}$, forms a copy of
    $M_j^-$. The $(j+2)$-set $C\cup\{w\}\cup\{a\}$ is a $j$-shell with
    base $C\cup \{w\}$ and apex vertex $a$.}
  \label{fig:mjstar}
\end{figure}
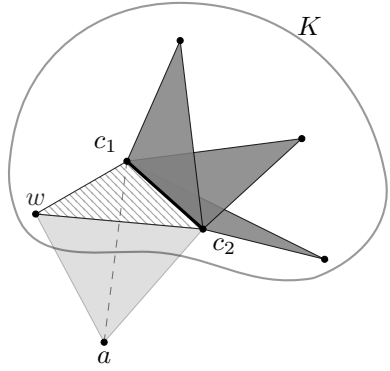

Observe that given a $4$-tuple $(K,C,w,a)$ which forms a copy of
$\Mjst$ in $\cG_p$, the $j$-shell $C \cup \{w\} \cup \{a\}$ is hollow
by \ref{Mjmin:flower} and the fact that
every $(j+1)$-simplex in $\cG_p$ is contained in a $k$-simplex.
Moreover, since the $j$-simplices of a $j$-shell form a $j$-cycle, a
copy of $\Mjst$ is in particular a copy of $\Mj$
(Definition~\ref{def:mj}). Therefore, the following
implications hold.
\begin{equation} \label{eq:Mj*implies}
  \Mjst \subset \cG_p \hspace{2ex} \Rightarrow \hspace{2ex}\Mj \subset
  \cG_p \hspace{2ex} \Rightarrow \hspace{2ex} \Mjmin \subset \cG_p.
\end{equation}
We will see later (Lemma~\ref{lem:jshells}) that for ``large'' $p$,
whp every copy of $\Mjmin$ is  extendable to several copies of $\Mjst$.
Therefore, the existence of copies of $\Mjmin$, $\Mjst$ and $\Mj$ in
$\cG_p$ are essentially equivalent events in that range.

Define $\varMjst$ to be the number of copies of $\Mjst$ in $\cG_p$.
We need a general expression for its expectation for certain possible
values of the probability $p$. To this end, consider the family
$\mathcal{T}^*$ of $4$-tuples $T^*=(K,C,w,a)$, where $K \subseteq[n]$
with $|K|=k+1$, where $C$ is a $j$-subset of $K$, where $w\in
K\setminus C$, and where $a \in [n] \setminus K$. Each of these tuples
may form a copy of $\Mjst$ with $K$ as $k$-simplex, $C$ as the centre,
and $C \cup \{w\} \cup \{a\}$ as the $j$-shell with base $C\cup\{w\}$
and apex vertex $a$. For each such tuple $T^*$, let $X_{T^*}$ be the
indicator random variable of the event that $T^*$ forms a copy of $\Mj$.

We next show that at probability $\pjmom$ the number of copies of
$\Mjst$ is concentrated around its expectation, whose order we also
determine.

\begin{lem}\label{lem:secondmoment}
  If $p=\pjmom$, then $\EE(\varMjst) =
  \Theta\left((\log n)^{j+2}\right)$. Furthermore, with high
  probability $\varMjst= (1 + o(1)) \EE(\varMjst)$.
\end{lem}

\begin{proof}
  Let $T^* = (K,C,w,a) \in \mathcal{T}^*$ be a fixed $4$-tuple. Recall
  that $T^*$ forms a copy of $\Mjst$ in $\cG_p$ if
  conditions~\ref{Mjst:simplex},~\ref{Mj:flower},
  and~\ref{Mjst:shell} of Definition~\ref{def:mjstar} hold.

  Clearly,~\ref{Mjst:simplex} holds with probability $p$. In order to
  determine the probability that~\ref{Mjst:flower} holds, consider a
  fixed petal. The probability that this petal lies in no other
  $k$-simplex is
  \begin{equation}\label{eq:r}
    r = r(p,n,k,j) := (1-p)^{\binom{n-j-1}{k-j}-1}. 
  \end{equation}
  For $p=\pjmom = \Theta\left(\frac{\log n}{n^{k-j+1}}\right)$, we
  have
  \begin{equation*} 
    r \ge 1 - \binom{n-j-1}{k-j}p = 1-o(1),
  \end{equation*}
  and thus each petal lies in no other $k$-simplices whp. Therefore,
  taking a union bound,~\ref{Mjst:flower} holds with probability at
  least $1-(k-j+1)(1-r) =1-o(1)$.

  Now consider~\ref{Mjst:shell}, conditioned on the event
  that both~\ref{Mjst:simplex} and~\ref{Mjst:flower} hold. The base 
  $C\cup\{w\}$ of $C\cup\{w\}\cup\{a\}$ already lies in $K$, 
  so it remains to prove
  that all other $(j+1)$-sets in $C\cup\{w\}\cup\{a\}$, i.e.\ the
  sides of this (potential) $j$-shell, are $j$-simplices in $\cG_p$.
  Denote the sides of $C\cup\{w\}\cup\{a\}$ by $L_1,\dotsc,L_{j+1}$.
  The number of $(k+1)$-sets containing $L_i$ is $\binom{n-j-1}{k-j}$,
  but some of these $(k+1)$-sets might not be allowed to be
  $k$-simplices because they contain a petal of the flower $\mathcal{F}(K,C)$
  (see~\eqref{eq:flowerinK}). However, the number of $(k+1)$-sets for
  which this is the case is $O(n^{k-j-1})$. All other $(k+1)$-sets
  meet $C\cup\{w\}\cup\{a\}$ only in $L_i$, which in particular
  implies that the events that $L_1,\dotsc,L_{j+1}$ lie in
  $k$-simplices (conditional on~\ref{Mjst:simplex}
  and~\ref{Mjst:flower}) are independent. Thus, each $L_i$ lies in a
  $k$-simplex independently with probability
  \begin{equation}\label{eq:side}
    1-(1-p)^{\binom{n-j-1}{k-j}+O(n^{k-j-1})} 
    = (1+o(1))q,
  \end{equation}
  where
  \begin{equation}\label{eq:q}
    q := \frac{pn^{k-j}}{(k-j)!} =
    \Theta\left(\frac{\log n}{n}\right).
  \end{equation}
  Therefore, conditional on~\ref{Mjst:simplex} and~\ref{Mjst:flower}
  holding,~\ref{Mjst:shell} holds with probability $(1+o(1))q^{j+1}$.
  The probability that $T^*$ forms a copy of $\Mjst$ is thus
  \begin{equation*}
    (1+o(1))pq^{j+1}.
  \end{equation*}

  The number of $4$-tuples $(K,C,w,a) \in \mathcal{T}^*$ is
  \begin{equation*}
    \binom{n}{k+1}\binom{k+1}{j}(k-j+1)(n-k-1) =
    (1+o(1))\frac{n^{k+2}}{j!(k-j)!}
  \end{equation*}
  and thus we have
  \begin{equation}\label{eq:expec}
    \EE(\varMjst) = (1+o(1))\frac{pq^{j+1}n^{k+2}}{j!(k-j)!} =
    \Theta\left((\log n)^{j+2}\right),
  \end{equation}
  as required.

  In order to prove the second statement of the lemma, we will show
  that $\EE(\varMjst^2) = (1+o(1))\EE(\varMjst)^2$ and then apply 
  Chebyshev's inequality. We have
  \begin{equation*}
    \EE(\varMjst^2) = 
    \sum_{T^*_1,T^*_2\in\mathcal{T}^*}
    \Pr \left( \{X_{T^*_1}=1\}\cap\{X_{T^*_2}=1\} \right).
  \end{equation*}
  Given two $4$-tuples $T^*_1=(K_1,C_1,w_1,a_1)$ and
  $T^*_2=(K_2,C_2,w_2,a_2)$, we define
  \begin{itemize}
  \item $I = I(T^*_1,T^*_2) :=
    (K_1 \cup \{a_1\}) \cap (K_2 \cup \{a_2\})$ and $i:=|I|$;
  \item $s = s(T^*_1,T^*_2) :=
    \begin{cases}
      1 & \mbox{if }K_1 = K_2,\\
      2 & \mbox{otherwise};
    \end{cases}$
  \item $\mathcal{L}_\ell$ to be the set of all $(j+1)$-subsets of
    $\{C_\ell\cup\{a_\ell\}\cup\{w_\ell\}\}$ for $\ell=1,2$ and 
    \begin{equation*}
      t=t(T^*_1,T^*_2) :=
      |(\mathcal{L}_1\cup\mathcal{L}_2)\setminus\{C_1\cup\{w_1\},C_2\cup\{w_2\}\}|,
    \end{equation*}
    i.e.\ the number of $(j+1)$-sets that are sides of the (potential)
    $j$-shells of $T^*_1$ and $T^*_2$, but not a base of either
    $j$-shell. 
  \end{itemize}
  If $s=2$ and the intersection of the two simplices contains a petal,
  then $T^*_1$ and $T^*_2$ cannot both form an $\Mjst$,
  because~\ref{Mjst:flower} would be violated. In the following, we
  therefore assume that this is not the case.

  The probability that both $T^*_1$ and $T^*_2$
  satisfy~\ref{Mjst:simplex} is $p^s$. As before,~\ref{Mjst:flower}
  holds whp. Conditioned on~\ref{Mjst:simplex} and~\ref{Mjst:flower}
  holding, we claim that~\ref{Mjst:shell} holds for both tuples
  simultaneously with probability $(1+o(1))q^t$. In order to prove
  this, denote the relevant sides of the two $j$-shells by
  $L_1,\dotsc,L_t$. No $k$-simplex can contain more than one side of
  the same $j$-shell, because otherwise it would also contain the base
  of the $j$-shell, which would contradict~\ref{Mjst:flower}. In
  particular, no $k$-simplex contains at least three of the $L_i$.
  Each $L_i$ lies in a $k$-simplex with probability $(1+o(1))q$
  by~\eqref{eq:side}. Moreover, the number of $(k+1)$-sets that
  contain $L_i\cup L_{i'}$ for some $i'\not=i$ is $O(n^{k-j-1})$.
  Thus, the probability that $L_i$ lies in such a $k$-simplex is
  \begin{equation*}
    1-(1-p)^{O(n^{k-j-1})} = O\left(\frac{\log n}{n^2}\right)
    \stackrel{\eqref{eq:q}}{=} o(q^2).
  \end{equation*}
  This means that the probability that $L_1,\ldots,L_t$ all lie in
  $k$-simplices is $(1+o(1))q^t$. This in turn yields
  \begin{equation}\label{eq:psqt}
    \Pr(\{X_{T^*_1}=1\} \cap \{X_{T^*_2}=1\}) = (1+o(1))p^s q^t.
  \end{equation}

  Define $\mathcal{T}^2(i,s,t)$ to be the set of pairs $(T_1^*,T_2^*)
  \in \mathcal{T}^*\times \mathcal{T}^*$ with parameters $i,s$ and
  $t$. Denote by $\mathcal{S}$ the set of triples $(i,s,t)$ for which
  $\mathcal{T}^2(i,s,t)$ is non-empty. With this
  notation,~\eqref{eq:psqt} implies that
  \begin{equation*}
    \EE(\varMjst^2) = (1+o(1))\sum_{(i,s,t) \in \mathcal{S}}\,
    \sum_{(T^*_1,T^*_2) \in \mathcal{T}^2(i,s,t)} p^sq^t.
  \end{equation*}
  Observe that $|\mathcal{T}^2(i,s,t)| = O(n^{2k+4-i})$. We can now
  estimate the contributions of all the summands, distinguishing the
  possible values of $s$ and $i$. 

  \emph{Case 1:} s=1. This means that $K_1 = K_2$ and thus $i\ge k+1$.
  \begin{itemize}
  \item $i=k+1$. In this case $a_1 \neq a_2$ and thus the sets of
    sides of the two $j$-shells would be disjoint, i.e.\ $t=2j+2$.
    Therefore we get a contribution of order
    \begin{equation*}
      O\left(pq^{2j+2}n^{2k+4-(k+1)}\right)
      \stackrel{\eqref{eq:expec}}{=}
      O\left(\frac{\EE(\varMjst)^2}{pn^{k+1}}\right)
      = o(\EE(\varMjst)^2).
    \end{equation*}
  \item $i=k+2$. The two $j$-shells have the same apex vertex and thus
    the $j$-shells coincide if and only if they have the same base.
    This means that $t\geq j+1$, which gives a contribution of order
    \begin{equation*}
      O(pq^{j+1}n^{2k+4-(k+2)}) \stackrel{\eqref{eq:expec}}{=}
      O\left( \EE(\varMjst) \right) =
      o\left( \EE(\varMjst)^2 \right).
    \end{equation*}
  \end{itemize}

  \emph{Case 2:} s=2.
  \begin{itemize}
  \item $i=0$. We show that this case represents the dominant
    contribution to $\EE(\varMjst^2)$. The two $j$-shells are
    disjoint, hence $t=2j+2$. Recall that we have
    \begin{equation*}
      (1+o(1))\frac{n^{k+2}}{j!(k-j)!}
    \end{equation*}
    choices for $T^*_1$. For any fixed $T^*_1$, the number of choices
    for $T^*_2$ that yield $i=0$ is
    \begin{equation*}
      \binom{n-k-1}{k+1}\binom{k+1}{j}(k-j+1)(n-2k-3) =
      (1+o(1))\frac{n^{k+2}}{j!(k-j)!}.
    \end{equation*}
    Thus, the contribution of all such pairs is
    \begin{equation*}
      (1+o(1))\frac{p^2q^{2j+2}n^{2k+4}}{(j!(k-j)!)^2}
      \stackrel{\eqref{eq:expec}}{=} (1+o(1)) \EE(\varMjst)^2.
    \end{equation*}
  \item $1\le i\le j$. In this case $T^*_1$ and $T^*_2$ cannot share a
    $j$-simplex of their shells, i.e.\ $t=2j+2$. Therefore the
    contribution is
    \begin{equation*}
      O\left(p^2q^{2j+2}n^{2k+4-i}\right)
      \stackrel{\eqref{eq:expec}}{=}
      O\left(\frac{\EE(\varMjst)^2}{n^i}\right) = o(\EE(\varMjst)^2).
    \end{equation*}
  \item $i=j+1$. Here, $T^*_1$ and $T^*_2$ can share at most one
    $j$-simplex of their shells, which means $t \geq 2j+1$ and we have
    a contribution of order
    \begin{equation*}
      O\left(p^2q^{2j+1}n^{2k+4-(j+1)}\right)
      \stackrel{\eqref{eq:expec}}{=}
      O\left(\frac{\EE(\varMjst)^2}{qn^{j+1}}\right)
      \stackrel{\eqref{eq:q}}{=} o(\EE(\varMjst)^2).
    \end{equation*}
  \item $j+2 \le i \le k+2$. In this case $t \ge j$, because $T^*_1$
    and $T^*_2$ may share their $j$-shells but have different bases,
    i.e.\ two $j$-simplices of the (potential) $j$-shells may be
    automatically present because of $K_1$ and $K_2$. Therefore the
    contribution is
    \begin{equation*}
      O\left(p^2q^jn^{2k+4-i}\right) \stackrel{\eqref{eq:expec}}{=}
      O\left(\frac{\EE(\varMjst)^2}{q^{j+2}n^i}\right)
      \stackrel{\eqref{eq:q}}{=} o(\EE(\varMjst)^2).
    \end{equation*}
  \end{itemize}
  Summing over all cases shows that $\EE(\varMjst^2) =
  (1+o(1))\EE(\varMjst)^2$, as desired. Thus, Chebyshev's inequality
  implies that $\varMjst = (1+o(1))\EE(\varMjst)$ whp.
\end{proof}

\begin{remark} \label{rem:atmostonestar}
  The case $s=1$, $i=k+1$ in the proof of Lemma~\ref{lem:secondmoment}
  also gives the expected number of pairs of copies of $\Mjst$ coming
  from a common $\Mjmin$. Since this is of order
  \begin{equation*}
    \Theta\left(pq^{2j+2}n^{k+3}\right) \stackrel{\eqref{eq:q}}{=}
    \Theta\left( \frac{(\log n)^{2j+3}}{n^j}\right) = o(1),
  \end{equation*}
  by Markov's inequality we deduce that whp in $\cG_{\pjmom}$ each 
  copy of $\Mjmin$ can be extended to \emph{at most} one copy of 
  $\Mjst$. We will make use of
  this observation in Lemma~\ref{lem:pj-1pjone}.
\end{remark}

In contrast to Remark~\ref{rem:atmostonestar}, the following lemma
ensures that at around $p=\pjone$, whp every $j$-simplex in $\cG_p$ is
the base of ``many'' $j$-shells. Thus it is very likely that each copy
of $\Mjmin$ gives rise to several copies of $\Mjst$, allowing us to
consider just copies of $\Mjmin$ as obstructions to \connectedness. In
other words,
\begin{equation*}
  \begin{array}{lcr} 
    \mbox{whp \quad for each } p\geq \pjone\mbox{,} & \hspace{2.5ex} &
    \Mjmin \subset \cG_p \hspace{1.5ex} \Rightarrow \hspace{1.5ex}
    \Mjst \subset \cG_p.
  \end{array}
\end{equation*}
Combining this with~\eqref{eq:Mj*implies}, the existence of copies of
$\Mjmin$, $\Mjst$ and $\Mj$ are essentially equivalent for
$p\ge\pjone$. Recall from Definition~\ref{def:addingsimplex} that for
a complex $\cG$ and a set $B$, $\cG + B$ is the complex obtained by
adding the set $B$ and its downward-closure to $\cG$.

\begin{lem}\label{lem:jshells}
  Let $p=\pjone$. Then there exists a positive constant $\gamma$ such
  that with high probability for every $(j+1)$-set $B$ the complex
  $\cG_p + B$ contains at least $\gamma n$ many $j$-shells that
  contain $B$. 
\end{lem}

\begin{proof}
  Recall that
  \begin{equation*}
    p=\pjone = \frac{1}{10(j+1)\binom{k+1}{j+1} n^{k-j} }.
  \end{equation*}

  Let $L_1,\ldots,L_{j+1}$ denote the $(j-1)$-simplices contained in
  $B$. We are interested in the number of vertices $a$ such that
  $B\cup\{a\}$ forms a $j$-shell, i.e.\ the number of $a \notin B$
  such that $L_i \cup \{a\}$ is a $j$-simplex in $\cG_p + B$ for all
  $i\in[j+1]$. To ensure independence in the following calculations,
  we will only consider a certain type of such $j$-shells, giving us a
  lower bound on their total number. Pick two disjoint sets $A$ and
  $D$ both of size $\lceil n/3 \rceil$ such that 
  $A \cap B = D \cap B = \emptyset$.
  We will consider only (potential) $j$-shells formed in the following
  way.
  \begin{itemize}
  \item The vertex $a$ is in $A$;
  \item for each $i=1,\ldots,j+1$, the $j$-simplex $L_i\cup\{a\}$ is
    present in $\cG_p$ (and thus also in $\cG_p + B$) as a subset of
    the $k$-simplex $R_i \cup L_i \cup \{a\}$, for some (not
    necessarily distinct) $(k-j)$-sets $R_1,\ldots,R_{j+1}$ in $D$.
  \end{itemize} 
  In this way all the required $j$-simplices would come from different
  $k$-simplices, ensuring independence.

  Fix $a \in A$ and let $E_a$ be the event that $B\cup\{a\}$ is a
  $j$-shell. Observe that for each $L_i$, the probability that there
  is no suitable set $R_i \subseteq D$ is
  \begin{equation*}
    (1-p)^{\binom{|D|}{k-j}} \leq
    (1-p)^{\frac{n^{k-j}}{4^{k-j}(k-j)!}}.
  \end{equation*}
  Therefore, setting $\beta:=10(j+1)\binom{k+1}{j+1}4^{k-j}(k-j)!$, by
  independence we have
  \begin{align*}
    \Pr(E_a)
    &\ge \left(1-(1-p)^{\frac{n^{k-j}}{4^{k-j}(k-j)!}}\right)^{j+1}\\
    &\ge \left(\frac{n^{k-j}}{4^{k-j}(k-j)!}p-\frac{1}{2}\left(
    \frac{n^{k-j}}{4^{k-j}(k-j)!}p\right)^2\right)^{j+1}\\
    &= \left(\frac{1}{\beta}-\frac{1}{2\beta^2}\right)^{j+1} =:
    \lambda > 0.
  \end{align*}
  The events $E_a$ are independent for distinct $a$, so the number of
  $j$-shells we count in this way dominates 
  $\mbox{Bi}(\lceil n/3 \rceil,\lambda)$.
  Fixing a constant $0< \gamma <\lambda/3$, we can apply the Chernoff
  bound (Lemma~\ref{lem:chernoff}) to deduce that
  \begin{equation*}
    \Pr\left( \mbox{Bi}(\lceil n/3 \rceil,\lambda) < \gamma n \right) \leq
    \exp\left(- \frac{(n\lambda/3 - \gamma n)^2}{2n\lambda/3} \right)
    = \exp\left(- \frac{n(\lambda/3 - \gamma)^2}{2\lambda/3} \right).
  \end{equation*}

  Finally, taking a union bound over all $\binom{n}{j+1}$ possible
  choices for the set $B$, we can bound the probability that the
  desired property does not hold by
  \begin{equation*}
    \binom{n}{j+1}\exp\left(-\frac{n(\lambda/3-\gamma)^2}{2\lambda/3}\right)
    = o(1),
  \end{equation*}
  as required.
\end{proof}

We now also prove that shortly before the (claimed) critical threshold for
\formalconnectedness, the number of copies of $\Mjmin$ is concentrated
around its expectation, using similar techniques as in
Lemma~\ref{lem:secondmoment}.

\begin{lem} \label{lem:meanconc}
  Let $\omega=o(\log n)$ be a function of $n$ which tends to infinity
  as $n \rightarrow \infty$. Let
  \begin{equation*}
    p\in\left[\pjmin,\frac{(j+1)\log n+\log\log n-\omega}{(k-j+1)n^{k-j}}(k-j)!\right],
  \end{equation*}
  and let $\varMjmin$ be the number of copies of $\Mjmin$ in $\cG_p$.
  Then $\EE(\varMjmin) = \Omega(e^\omega)$ and with high probability
  $\varMjmin = (1+o(1))\EE(\varMjmin)$.
\end{lem}

\begin{proof}
  Let $K$ be a $(k+1)$-set and let $C$ be a $j$-set in $K$. In order
  for $(K,C)$ to form a copy of $M_j^-$, we need $K$ to be a
  $k$-simplex and each petal of the flower $\mathcal{F}(K,C) =
  \{C\cup\{w\} \mid w\in K\setminus C\}$ to lie in no other
  $k$-simplex. For a fixed petal, the probability of this event is
  equal to $r=(1-p)^{\binom{n-j-1}{k-j}-1}$ defined in~\eqref{eq:r}.
  Moreover, there are $O(n^{k-j-1})$ many $(k+1)$-sets that contain
  more than one petal. Now since
  \begin{equation*}
    (1-p)^{O(n^{k-j-1})} = 1-o(1),
  \end{equation*}
  whp there are no $k$-simplices containing more than one petal. Thus,
  \begin{align} \label{eq:expMjminus}
    \EE(\varMjmin) &=
    (1+o(1))\binom{n}{k+1} \binom{k+1}{j} p r^{k-j+1} \nonumber\\ 
    &= (1+o(1))\binom{n}{k+1}\binom{k+1}{j}p(1-p)^{(k-j+1)\binom{n-j-1}{k-j}}.
  \end{align}

  The derivative of the right hand side of~\eqref{eq:expMjminus} with
  respect to $p$ is negative throughout the considered interval.
  Therefore the upper extreme of $p$ gives the smallest expectation,
  which is of order
 
  \begin{equation*}
    \Theta(n^{k+1}) \Theta\left(\frac{\log n}{n^{k-j}}\right)
    \Theta\left( \exp \left(-(j+1)\log n -\log\log n 
    +\omega \right)\right) =
    \Theta\left(e^\omega\right) \rightarrow \infty.
  \end{equation*}
  
  In order to apply a second moment argument, we will now show that
  \begin{equation*}
    \EE(\varMjmin^2) = (1+o(1))\EE(\varMjmin)^2,
  \end{equation*} 
  implying that whp $\varMjmin$ is concentrated around its
  expectation. Let $\mathcal{T}^-$ denote the family of pairs
  $T^-=(K,C)$, where $K \subseteq [n]$ with $|K|=k+1$ and $C$ is a
  $j$-subset of $K$. Each of these pairs may form a copy of $\Mjmin$
  with $K$ as $k$-simplex and $C$ as centre of the flower
  $\mathcal{F}(K,C)$.

  Given two pairs $T^-_1=(K_1,C_1)$ and $T^-_2=(K_2,C_2)$, we define
  \begin{itemize}
  \item $s = s(T^-_1,T^-_2) :=
    \begin{cases}
      1 & \mbox{if } K_1 = K_2,\\
      2 & \mbox{otherwise};
    \end{cases}$
  \item $\mathcal{F}_\ell := \mathcal{F}(K_\ell,C_\ell)$ for
    $\ell =1,2$;
  \item $t=t(T^-_1,T^-_2) := |\mathcal{F}_1 \cup \mathcal{F}_2|$, 
    i.e.\ the total number of (potential) petals. 
  \end{itemize}
  The probability of two pairs in $\mathcal{T}^-$ both forming a copy 
  of $\Mjmin$ is $(1+o(1))p^sr^t$. With this observation, we can
  determine the contribution to $\EE(\varMjmin^2)$ of the pairs with a
  fixed value of $s$.
  \begin{itemize}
  \item $s=1$. Petals can be shared, but certainly $t\geq k-j+1$ and
    the contribution is at most of order
    \begin{equation*}
      O\left(n^{k+1} p r^{k-j+1}\right)
      \stackrel{\eqref{eq:expMjminus}}{=} O(\EE(\varMjmin)) =
      o(\EE(\varMjmin)^2).
    \end{equation*}
  \item $s=2$. By definition, a petal cannot lie in any other
    $k$-simplex and thus only the pairs with $t=2(k-j+1)$ have a
    positive probability of both forming a copy of $\Mjmin$. The
    number of such pairs is
    \begin{equation*}
      \binom{n}{k+1} \binom{n-k-1}{k+1} \binom{k+1}{j}^2 + O(n^{2k+1})
       = (1+o(1))\binom{n}{k+1}^2\binom{k+1}{j}^2.
    \end{equation*}
    Thus these pairs provide a contribution of
    \begin{equation*}
      (1+o(1))\binom{n}{k+1}^2\binom{k+1}{j}^2p^2r^{2(k-j+1)}
      \stackrel{\eqref{eq:expMjminus}}{=} (1+o(1))\EE(\varMjmin)^2.
    \end{equation*}
  \end{itemize} 
  In total, we have $\EE(\varMjmin^2) =
  (1+o(1))\EE(\varMjmin)^2$, and Chebyshev's inequality implies that
  $\varMjmin = (1+o(1))\EE(\varMjmin)$ whp.
\end{proof}

\subsection{Excluding obstructions and determining the hitting
  time}\label{sec:subcritical:excluding}

The goal of this section is to determine when there are no more copies
of $\Mj$ in $\cG_p$ whp. This result, together with
Lemmas~\ref{lem:jshells} and~\ref{lem:meanconc}, will enable us to
prove that whp the birth time $\pMj$ is close to $\pj$,
the (claimed) threshold for \connectedness\ (Corollary~\ref{cor:pMj}).

Consider the probability 
\begin{equation}\label{eq:pjbar}
  \bar{p}_j :=
  \frac{(j+1)\log n + \frac{1}{2}\log\log n}{(k-j+1) n^{k-j}}(k-j)!.
\end{equation}
Define $\pstMj$ as the first birth time $p$ larger than $\bar{p}_j$
such that there are no copies of $\Mj$ in $\cG_p$. By
Lemmas~\ref{lem:jshells} and~\ref{lem:meanconc}, whp $\cG_{\bar{p}_j}$
contains a growing number of copies of $\Mj$. By definition of $\pMj$,
conditioned on this high probability event we have $\pstMj \leq \pMj$.
In the next lemma we show that in fact they are equal whp. To do so,
we need the following definition.

\begin{definition} \label{def:localobstacle}
  Given a $k$-complex $\cG$, a $k$-simplex $K$ is a \emph{local
  obstacle} if $K$ contains at least $k-j+1$ many $j$-simplices which
  are not contained in any other $k$-simplex of $\cG$.
\end{definition}

Note that this definition is similar to that of $\Mjmin$
(Definition~\ref{def:mjminus}), but without the restriction that the
$k-j+1$ many $j$-simplices must form a flower.

\begin{lem}\label{lem:localobstacle}
  With high probability, for all $p\geq \bar{p}_j$ every local
  obstacle that exists in $\cG_p$ also exists in $\cG_{\bar{p}_j}$.
  In particular, we have $\pMj = \pstMj$ whp.
\end{lem}

\begin{proof}
  Suppose that $\cG_p$ contains a local obstacle which is not present
  in $\cG_{\bar{p}_j}$ and let $K$ be the $(k+1)$-set realising this
  obstacle. Then its birth time $p_K$ satisfies $p_K\in(\bar{p}_j,p]$.
  The set $K$ can become a local obstacle only if

  \begin{enumerate}
  \item\label{enum:local}
    $K$ contains a collection $\mathcal{L}$ of (at least)
    $k-j+1$ many $(j+1)$-sets which are not yet
    $j$-simplices in $\cG_{\bar{p}_j}$;
  \item\label{enum:birth}
    $p_K$ is smaller than the birth time of any other $(k+1)$-set
    containing at least one of the $(j+1)$-sets in $\mathcal{L}$.
  \end{enumerate}
  
  If $K$ satisfies~\ref{enum:local}, then for any $(j+1)$-set $L 
  \in \mathcal{L}$, no $(k+1)$-set intersecting $K$ precisely in
  $L$ is allowed to be a $k$-simplex in $\cG_{\bar{p}_j}$ and thus
  there are at least $\binom{n-k-1}{k-j}(k-j+1)$ many $(k+1)$-sets
  which are not $k$-simplices. Hence, given $K$ and $k-j+1$ fixed
  $(j+1)$-sets within $K$, the probability of satisfying
  property~\ref{enum:local} in $\cG_{\bar{p}_j}$ is bounded from above
  by
  \begin{align*}
    (1-\bar{p}_j)^{\binom{n-k-1}{k-j}(k-j+1)} &=
    (1+o(1))\exp\left(-\log\left(n^{j+1}\right)-\log\left((\log n)^{1/2}\right)\right)\\
    &= O\left(\frac{1}{n^{j+1}\sqrt{\log n}}\right).
  \end{align*}
  On the other hand, each $(j+1)$-set in $\mathcal{L}$ is
  contained in $\binom{n-j-1}{k-j}$ potential $k$-simplices. In order
  for $K$ to satisfy~\ref{enum:birth}, all those $k$-simplices would
  need to have larger birth time than $K$, which happens with
  probability $O\left(\frac{1}{n^{k-j}}\right)$. Thus, the expected
  number of sets $K$ satisfying~\ref{enum:local} and~\ref{enum:birth}
  is at most
  \begin{equation*}
    \binom{n}{k+1} 2^{\binom{k+1}{j+1}}
    O\left(\frac{1}{n^{j+1}\sqrt{\log n}}\right)
    O\left(\frac{1}{n^{k-j}}\right) =
    O\left(\frac{1}{\sqrt{\log n}}\right) = o(1)
  \end{equation*}
  and the conclusion follows by Markov's inequality.   
\end{proof}

Observe that in particular each copy of $\Mjmin$ is a local obstacle.
Thus, we derive the following corollary.

\begin{cor}\label{cor:noMjmin}
  Whp for all $p\ge\pMj$, there are no copies of $\Mjmin$ in $\cG_p$.
\end{cor}

We can now easily deduce that the birth time $\pMj$ at which the last
copy of $M_j$ disappears is close to $\pj$. Observe that the following
corollary is exactly
Theorem~\ref{thm:gentheor}~\ref{thm:gentheor:pMj}.

\begin{cor}\label{cor:pMj}
  Let $\omega$ be any function of $n$ which tends to infinity as $n$
  tends to infinity. Then whp
  \begin{equation*}
    \frac{(j+1)\log n+\log\log n-\omega}{(k-j+1)n^{k-j}}(k-j)! < \pMj
    < \frac{(j+1)\log n+\log\log n+\omega}{(k-j+1)n^{k-j}}(k-j)!.
  \end{equation*}
\end{cor}

\begin{proof}
  We may assume without loss of generality that $\omega=o(\log n)$. By
  Lemmas~\ref{lem:jshells} and~\ref{lem:meanconc}, $\pMj >
  \frac{(j+1)\log n+\log\log n-\omega}{(k-j+1)n^{k-j}}(k-j)!$ whp. On
  the other hand, setting $p =
  \frac{(j+1)\log n + \log\log n + \omega}{(k-j+1) n^{k-j}}(k-j)!$ and
  arguing as in Lemma~\ref{lem:meanconc} (see~\eqref{eq:expMjminus}),
  the expected number of copies of $\Mjmin$ is bounded from above by 
  \begin{align*}
    (1+o(1))&n^{k+1}p\exp\left(-\frac{(n-j-1)^{k-j}}{(k-j)!}(k-j+1)p\right)\\
    &= \Theta \left( n^{j+1}\log n\exp\left(-(j+1)\log n-\log\log n
    -\omega \right) \right)\\
    &= \Theta \left( e^{-\omega} \right) = o(1).
 \end{align*}
  So by Markov's inequality, whp there are no copies of $\Mjmin$ and
  thus also no copies of $\Mj$ in $\cG_p$, i.e.
  \begin{equation*}
    \pstMj < \frac{(j+1)\log n+\log\log n+\omega}{(k-j+1)n^{k-j}}(k-j)!
  \end{equation*}
  and by Lemma~\ref{lem:localobstacle} we have $\pstMj  = \pMj$ whp.
\end{proof}

\subsection{Covering the intervals: proof of
  Lemma~\ref{lem:phantomlemma}}\label{sec:subcrit:covering}

In order to prove Lemma~\ref{lem:phantomlemma}, we show that for each
$j\in[k-1]$, whp there exist three minimal obstructions which survive
throughout each of the intervals $[\pjmom,\pjone]$, $[\pjone,\pjmin]$
and $[\pjone,\pMj)$, respectively. 

Recall that 
\begin{equation}\label{eq:pjone}
  \pjone=\displaystyle \frac{1}{10(j+1)\binom{k+1}{j+1} n^{k-j}}.
\end{equation}  
The first step is to show that at least one of the $\varMjst =
\Theta((\log n)^{j+2})$ copies of $\Mjst$ which are present whp at
probability $\pjmom$ (Lemma~\ref{lem:secondmoment}) survives until
time $\pjone$. To do so, we will count the number of \emph{dangerous
sets}, that is $(k+1)$-sets which, if they are selected as
$k$-simplices, make one or more of copies of $\Mjst$ disappear. Then
we show that whp up to probability $\pjone$ the number of copies of
$\Mjst$ destroyed by dangerous sets which became $k$-simplices is less
than $\varMjst$.

\begin{lem}\label{lem:pj-1pjone}
  With high probability one copy of $\Mjst$ exists in $\cG_p$
  throughout the range $[\pjmom,\pjone]$.
\end{lem}

\begin{proof}
  Define $x=2\EE(\varMjst)$ at time $p=\pjmom$. By
  Lemma~\ref{lem:secondmoment}, we know that whp
  \begin{equation*}
    \frac{x}{3} \leq \varMjst \leq x,
  \end{equation*}
  so let us condition on this high probability event occurring.

  We know that we can generate $\cG_{\pjone}$ from $\cG_{\pjmom}$ by
  exposing an additional probability of $\frac{\pjone-\pjmom}{1-\pjmom}
  \leq \pjone$, therefore we will use the upper bound $\pjone$ in the
  following calculations. Set $p=\pjone$ and let $Y$ be the number of
  dangerous sets selected as $k$-simplices in $\cG_p$. A $(k+1)$-set
  can contain at most $\binom{k+1}{j+1}$ petals, each of which can be
  part of at most $j+1$ different copies of $\Mjmin$, since by
  definition a petal belongs to exactly one $k$-simplex and within
  this petal we have $\binom{j+1}{j}=j+1$ choices for the centre which
  then uniquely defines the copy of $\Mjmin$. So each of the
  $k$-simplices counted by $Y$ can destroy at most $c:=
  \binom{k+1}{j+1}(j+1)$ copies of $\Mjmin$. Moreover, by
  Remark~\ref{rem:atmostonestar}, whp $c$ is also the maximum number
  of copies of $\Mjst$ that can disappear by adding a dangerous set to
  the complex. Therefore, we now show that
  \begin{equation*}
    \Pr\left(cY \geq \frac{x}{3}\right) = o(1).
  \end{equation*}
  This will imply that whp $cY < \varMjst$, so at least one of the
  copies of $\Mjst$ counted by $\varMjst$ will survive throughout the
  considered probability interval.
  
  A dangerous $(k+1)$-set makes one or more copies of $\Mjst$
  disappear if it becomes a $k$-simplex and contains at least one
  petal of each of their flowers. For a copy of $\Mjst$, the number of
  $(k+1)$-sets that intersect it in at least one petal is at most
  $(k-j+1) \binom{n-j-1}{k-j}$. Therefore, whp the number of dangerous
  $(k+1)$-sets is at most 
  \begin{equation*}
    (k-j+1) \binom{n-j-1}{k-j} x \leq \frac{k-j+1}{(k-j)!} n^{k-j} x
    \leq 2 n^{k-j} x =: N.
  \end{equation*}

  Due to the independence of the chosen $k$-simplices, $Y$ is
  dominated by $\mbox{Bi}(N,p)$. Since 
  \begin{equation*}
    \EE(\mbox{Bi}(N,p)) = Np
    \stackrel{\eqref{eq:pjone}}{=} \frac{x}{5c},
  \end{equation*}
  by the Chernoff bound (Lemma~\ref{lem:chernoff}) we have
  \begin{align*}
    \Pr \left( Y \geq \frac{x}{3c}\right) &\leq
    \Pr\left( \mbox{Bi}(N,p)\geq \frac{x}{3c} \right)\\
    &\leq \exp\left(-\frac{\left(\frac{x}{3c}-Np\right)^2}{2\left(Np+\left(\frac{x}{3c}-Np\right)/3\right)}\right)\\
    &= \exp\left(-\frac{2x}{55c}\right) = o(1),
  \end{align*}
  because $x\xrightarrow{n\rightarrow \infty} \infty$ by
  Lemma~\ref{lem:secondmoment}.
\end{proof}
We now consider the second subinterval $[\pjone,\pjmin]$. In this
range, we will show that whp one of the ``many'' copies of $\Mjmin$
which exist whp at time $\pjmin$ (Lemma~\ref{lem:meanconc}) was
already present at the beginning of the interval. Together with the
fact that whp each $\Mjmin$ gives rise to a copy of $\Mj$
(Lemma~\ref{lem:jshells}), this will imply that whp one copy of $\Mj$
exists throughout this interval.

\begin{lem}\label{lem:pj1pjmin}
  With high probability one copy of $\Mjmin$ exists in $\cG_p$
  throughout the range $[\pjone, \pjmin]$.
\end{lem}

\begin{proof}
  Set
  \begin{equation*}
    p = \pjmin =
    \left(1-\frac{1}{\sqrt{\log n}}\right)\frac{(j+1)\log n}{(k-j+1)n^{k-j}}(k-j)!.
  \end{equation*}
  By Lemma~\ref{lem:meanconc}, at probability $p$ the number
  $\varMjmin$ of copies of $\Mjmin$ is concentrated around its
  expectation
  \begin{equation*}
    \EE(\varMjmin) \stackrel{\eqref{eq:expMjminus}}{=}
    \Theta\left(n^{k+1}p(1-p)^{(k-j+1)\binom{n-j-1}{k-j}}\right) =
    \Theta\left(n^{\frac{j+1}{\sqrt{\log n}}}\log n\right),
  \end{equation*}
  which is growing with $n$. Note that a fixed $k$-simplex can give
  rise to only $\binom{k+1}{j} = \Theta(1)$ different copies of
  $\Mjmin$. Therefore whp there are
  $\Theta\left(n^{\frac{j+1}{\sqrt{\log n}}}\log n\right)$ many copies
  of $\Mjmin$ that arise from different $k$-simplices, and whose birth
  times are thus independent. Given that these copies exist at time
  $\pjmin$, the birth times of the corresponding $k$-simplices are
  uniformly distributed in the interval $[0,\pjmin]$. The probability
  that any fixed such copy already existed at time $\pjone$ is therefore 
  \begin{equation*}
    \frac{\pjone}{\pjmin} = \Theta\left(\frac{1}{\log n}\right).
  \end{equation*}
  Thus, because of the independence, the probability that none of them
  was present at $\pjone$ is at most
  \begin{equation*}
    \left(1-\Theta\left(\frac{1}{\log n}\right)\right)^{\Theta\left(n^{\frac{j+1}{\sqrt{\log n}}}\log n\right)}
    \leq \exp\left(-\Theta\left(n^{\frac{j+1}{\sqrt{\log n}}}\right)\right)
    = o(1),
  \end{equation*}
  as required.
\end{proof}

We now conclude the argument by covering the third interval
$[\pjmin,\pMj)$ of the subcritical range.

\begin{lem} \label{lem:pjmin-pMj}
  With high probability one copy of $\Mjmin$ exists in $\cG_p$
  throughout the range $[\pjmin, \pMj)$.
\end{lem}

\begin{proof}
  By the definition of $\pjmin$ and Corollary~\ref{cor:pMj}, we know
  that whp $\pjmin = (1 -o(1)) \pMj$. So, conditioning on 
  this high probability event and arguing as in the proof of
  Lemma~\ref{lem:pj1pjmin}, the final minimal obstruction to disappear
  at time $\pMj$ already existed at time $\pjmin$ with probability at
  least 
  \begin{equation*}
    \frac{\pjmin}{\pMj}= 1 - o(1),
  \end{equation*}
  as required.
\end{proof}

\begin{proof}[Proof of Lemma~\ref{lem:phantomlemma}]
  By Lemma~\ref{lem:jshells}, the copies of $\Mjmin$ from
  Lemmas~\ref{lem:pj1pjmin} and~\ref{lem:pjmin-pMj} whp give rise to
  copies of $\Mjst$, and thus in particular to copies of $\Mj$.
  Therefore, Lemmas~\ref{lem:pj-1pjone},~\ref{lem:pj1pjmin}
  and~\ref{lem:pjmin-pMj} together imply Lemma~\ref{lem:phantomlemma}.
\end{proof}

\section{Critical window and supercritical regime} \label{sec:supcrit}

\subsection{Overview}

In this section, we study obstructions around the point of the claimed
phase transition and in the supercritical regime, that is, for $p =
(1+o(1))p_j$ and $p\ge\pMj$, respectively. The results of this section
will form the foundation of the proof of
Theorem~\ref{thm:gentheor}~\ref{thm:gentheor:supercrit}. Furthermore,
they will be an essential ingredient in the proof of
Theorem~\ref{thm:critwindow1}.

By the definition of $\pMj$, there are no copies of $\Mj$ in $\cG_p$
(and whp also no copies of $\Mjmin$ by Corollary~\ref{cor:noMjmin})
for any $p\ge\pMj$. It remains to show that there are no other
obstructions either. In fact, we shall even prove
(Corollary~\ref{cor:supercritical}) that from slightly before $\pMj$
onwards, all $j$-cocycles are generated by copies of $\Mjmin$ (recall
that a $j$-cocycle is a $j$-cochain in $\ker\delta^{j}$, see
Section~\ref{sec:preliminaries:cohomology}). To make this more
precise, we need the following terminology.
\begin{definition}\label{def:generating}
  Let $(K,C)$ be a copy of $\Mjmin$ in a $k$-complex $\mathcal{G}$. We
  say that a $j$-cochain $f_{K,C}$ \emph{arises from $(K,C)$} if its
  support is the $j$-flower $\mathcal{F}(K,C)$. (Observe that
  $f_{K,C}$ is then a $j$-cocycle.)
  
  We say that a $j$-cocycle $f$ in $\mathcal{G}$ is \emph{generated by
  copies of $\Mjmin$} if it lies in the same cohomology class as a sum
  of $j$-cocycles that arise from copies of $\Mjmin$. We denote by
  $\mathcal{N}_{\cG}$ the set of $j$-cocycles that are \emph{not}
  generated by copies of $\Mjmin$.
\end{definition}

Our goal is to show that whp $\mathcal{N}_{\cG_p}=\emptyset$ for
$p\ge\pjmin$ (Corollaries~\ref{cor:noSpsmallp}~and~\ref{cor:supercritical}),
which in particular
will imply that whp each $j$-cocycle in $\cG_p$ is also a
$j$-coboundary (i.e.\ there are no bad functions, see
Definition~\ref{def:badfctn}) for all $p\ge\pMj$. Furthermore, it
will enable us to directly relate the number of copies of $\Mjmin$ to
the dimension of $H^j(\cG_p;\FF_2)$ (cf.\ Theorem~\ref{thm:critwindow1}).

\begin{definition}\label{def:f_p}
  For each $p\in [0,1]$, let $f_p$ be a function in
  $\mathcal{N}_{\cG_p}$ with smallest support $S_p$, if such a
  function exists.
\end{definition}

In order to prove that whp $\mathcal{N}_{\cG_p}$ is empty, we show
(Lemma~\ref{lem:traversable}) that for any $k$-complex $\cG$, a
smallest support of elements of $\mathcal{N}_{\cG}$ (and so in
particular $S_p$ in $\cG_p$) would have to be \emph{traversable} (see
Definition~\ref{def:traversability}). We then show that whp no $\cG_p$
with $p\geq \pjmom$ can contain a traversable support $S_p$. For
``small'' sizes of $S_p$ and $p=(1+o(1))\pj$, basic estimates and a
union bound argument will suffice (Lemma~\ref{lem:smallsupport}); for
larger size, we will make use of traversability to define a
breadth-first search process that finds all possible supports. In this
way, we can bound the number of possibilities for $S_p$ more
carefully, thus allowing us to prove that whp for all relevant $p$
simultaneously, $S_p$ cannot be ``large''
(Lemma~\ref{lem:largesupport}). Finally, we complete the argument
proving that whp no new elements of $\mathcal{N}_{\cG_p}$ with ``small''
support size can appear if we increase $p$ (Lemma~\ref{lem:nomoreMj}).

\subsection{Traversability}

\begin{definition}\label{def:traversability}
  Let $\cG$ be a $k$-complex in which each simplex is contained in a
  $k$-simplex, and let $S$ be a collection of $j$-simplices of $\cG$.
  For $\sigma_1,\sigma_2 \in S$, we set
  \begin{center}
    \begin{tabular}{lcr}
      $\sigma_1 \sim \sigma_2$ & if & $\sigma_1$ and $\sigma_2$ lie in
      a common $k$-simplex.
    \end{tabular}
  \end{center}
  We say that the set $S$ is \emph{traversable} if the transitive
  closure of $\sim$ is $S \times S$.
\end{definition}

In other words, a set of $j$-simplices in such a $k$-complex is
traversable if it \emph{cannot} be partitioned into two non-empty
subsets such that each $k$-simplex (and thus also each
$(j+1)$-simplex) contains $j$-simplices in at most one of the two
subsets.

\begin{lem}\label{lem:traversable}
  Let $\cG$ be a $k$-complex in which each simplex is contained in a
  $k$-simplex, and let $f$ be an element of $\mathcal{N}_{\cG}$ with
  smallest support $S$. Then $S$ is traversable. In particular, $S_p$
  is traversable in $\cG_p$, if it exists, for each $p\in [0,1]$.
\end{lem}

\begin{proof}
  Suppose $S$ is not traversable. Then we can find a partition
  $S =  T_1 \mathop{\dot{\cup}} T_2$, with $T_1$ and $T_2$ 
  non-empty such that each
  $(j+1)$-simplex of $\cG$ contains $j$-simplices in at most one of
  the two parts. Define $g_1$ and $g_2$ to be $j$-cochains
  with supports $T_1$ and $T_2$, respectively. By the choice of
  $T_1$ and $T_2$, both $g_1$ and $g_2$ are $j$-cocycles. Moreover,
  neither of them lies in $\mathcal{N}_{\cG}$ by the minimality of
  $S$. As the property of being generated by copies of $\Mjmin$ is
  closed under summation, $f = g_1+g_2$ is generated by copies of
  $\Mjmin$, a contradiction to $f\in\mathcal{N}_{\cG}$.
\end{proof}

\subsection{Small supports}

The following counting argument shows that whp, at around time $\pj$
traversable supports of $j$-cocycles of constant size do not exist.
This implies in particular that $S_p$ (if it exists) has to be
``large''.
\begin{lem}\label{lem:smallsupport}
  For $p= (1+o(1))\pj$ and for any constant $d\geq k-j+2$, with high
  probability there is no $j$-cocycle in $\cG_p$ with traversable
  support of size $s$ with $k-j+2\leq s \leq d$. In particular, with
  high probability either $S_p$ does not exist or $|S_p|>d$.
\end{lem}

\begin{proof}
  Consider a traversable support $S$ of a $j$-cocycle of size $s$
  with $k-j+2\leq s \leq d$.
  Suppose that $S$ covers $v$ vertices and denote by $\ell$ the number
  of $k$-simplices that make $S$ traversable. These quantities are
  easily bounded by
  \begin{equation}\label{eq:boundl}
    \frac{s}{\binom{k+1}{j+1}} \leq \ell \leq s \le d,
  \end{equation}
  and
  \begin{equation}\label{eq:boundv}
    v \leq j+1+(k-j)\ell.
  \end{equation}

  We know by Lemma~\ref{lem:minobst} that if a $k$-simplex contains a
  $j$-simplex in $S$, then all its $k+1$ vertices are covered by $S$.
  Therefore, all $s \binom{n-v}{k-j}$ many $(k+1)$-sets consisting of
  the vertices of one $j$-simplex in $S$ and $k-j$ vertices not
  covered by $S$ cannot be $k$-simplices in $\cG_p$. Thus, the
  probability that a fixed such $S$ exists is at most
  \begin{align*}
    p^\ell (1-p)^{s\binom{n-v}{k-j}} &= p^\ell
    (1-p)^{s \left(\frac{n^{k-j}}{(k-j)!} + O(n^{k-j-1})\right)}\\
    &= O\left(\left(\frac{\log n}{n^{k-j}}\right)^\ell
    \exp\left(-\frac{s(j+1)}{k-j+1}\log n + o(\log n) \right)\right)\\
    &= O\left(n^{-\ell(k-j) - \frac{s(j+1)}{k-j+1} + o(1)}
    \left( \log n \right)^\ell\right). 
  \end{align*}
  Denote by $E_{s,v,\ell}$ the event that a traversable support $S$
  with parameters $s$, $v$, and $\ell$ exists. There are $O(n^v)$
  different ways of choosing $S$, thus
  \begin{equation*}
    \Pr(E_{s,v,\ell}) =
    O\left(n^{v-\ell(k-j)-\frac{s(j+1)}{k-j+1}+o(1)}
    (\log n)^\ell\right).
  \end{equation*}

  Using~\eqref{eq:boundv} and the fact that $s\geq k-j+2$, we obtain
  \begin{equation*}
    v-\ell(k-j) - \frac{s(j+1)}{k-j+1} + o(1) \le
    -\frac{j+1}{k-j+1} + o(1) \le -\frac{j}{k-j+1}
  \end{equation*}
  and thus
  \begin{equation*}
    \Pr(E_{s,v,\ell}) = o(1).
  \end{equation*}
  Finally, observe that by~\eqref{eq:boundl} and~\eqref{eq:boundv},
  there is only a constant number of possible values for $s$, $v$, and
  $\ell$. Therefore, the probability that any such support $S$ exists
  is $o(1)$, as required.
\end{proof}
Note that a similar argument also works for $s$ up to
$O\left(\frac{\log n}{\log \log n}\right)$, but we only need it for
constant size, since we will cover the range between constant size and
size $O\left(\frac{\log n}{\log \log n}\right)$ with a different
argument that we use for all large $s$.

\subsection{Large supports}

For larger support sizes, the previous calculations do not work
anymore and we will need a more careful technique for bounding the
number of possible supports, namely a breadth-first search process. We
will also make use of the following proposition due to Meshulam and
Wallach~\cite{MeshulamWallach08}.

\begin{prop}[{\cite[Proposition~3.1]{MeshulamWallach08}}]\label{prop:meshwal}
  Let $\Delta$ be the downward-closure of the $(n-1)$-simplex on
  vertex set $[n]$, where $n\ge j+2$. For $f\in C^j(\Delta)$, define
  $w(f)$ to be the smallest size of a support of a $j$-cochain of the
  type $f+\delta^{j-1}g$, where $g\in C^{j-1}(\Delta)$. Furthermore,
  denote by $b(f)$ the size of the support of $\delta^jf$, i.e.\ the
  number of $(j+1)$-simplices in $\Delta$ containing an odd number of
  $j$-simplices of the support in $f$. Then
  \begin{equation*}
    b(f) \geq \frac{w(f)n}{j+2}.
  \end{equation*}
\end{prop}

In the next lemma we show that whp in the supercritical range, a 
smallest support of elements of $\mathcal{N}_{\cG_p}$ cannot be
``large''.
 
\begin{lem}\label{lem:largesupport}
  There exists a positive constant $\bar{d}$ such that with high
  probability for all $p\geq \pjmin$, either $S_p$ does not exist or
  $|S_p| < \bar{d}$.
\end{lem}

\begin{proof}
  Write $s := |S_p|$. By Lemma~\ref{lem:traversable}, $S_p$ (if it
  exists) is traversable and thus we can discover it via the following
  breadth-first search process: start from any $j$-simplex in $S_p$
  and query all $(k+1)$-sets containing it. Since $S_p$ is the support
  of the $j$-cocycle $f_p$, any of these sets which forms a
  $k$-simplex must contain at least one other $j$-simplex in $S_p$.
  From all $j$-simplices in $S_p$ found in this way, we can continue
  the process according to some pre-determined order of $j$-simplices,
  but we explore only $(k+1)$-sets which would give us some previously
  undiscovered $j$-simplex in $S_p$. By the traversability of $S_p$,
  we discover all of $S_p$ in this process. 

  Let us bound the number of traversable supports of size $s$ which
  are contained in $\ell\le s$ many $k$-simplices
  (recall~\eqref{eq:boundl}), which we can find via the described
  search process. Define the sequence $\underline{b} =
  (b_1,\ldots,b_s)$, where $b_i\geq 0$ is the number of
  $k$-simplices we discover from the $i$-th $j$-simplex in this
  process. From the $i$-th $j$-simplex we may query up to
  $\binom{n}{k-j}$ many $(k+1)$-sets and for each of the $b_i$
  discovered $k$-simplices we can find at most $\binom{k+1}{j+1}-1$
  new $j$-simplices of the support, so this can happen in at most
  $\binom{\binom{n}{k-j}}{b_i}2^{\binom{k+1}{j+1}b_i}$ different ways.
  Thus, if we condition on the sequence $\underline{b}$, the number of
  supports of size $s$ we can find is bounded from above by
  \begin{equation*}
    \binom{n}{j+1} \prod_{i=1}^{s}\binom{\binom{n}{k-j}}{b_i}
    2^{\binom{k+1}{j+1}b_i} \le n^{j+1}\frac{\left(\binom{n}{k-j}
    2^{\binom{k+1}{j+1}}\right)^\ell }{\prod_{i=1}^{s} b_i!},
  \end{equation*}
  where we are using that $\sum_{i=1}^s b_i = \ell$.

  In order to apply Proposition~\ref{prop:meshwal} to $f_p$ (which is
  possible, because $\cG_p$ is a sub-complex of $\Delta$), let us
  determine the value $w(f_p)$. First observe that for $p\ge\pjmin$,
  whp $\cG_p$ has a complete $(j-1)$-dimensional skeleton, which can
  be proved by a simple first moment calculation. Thus, if we consider
  $f_p+\delta^{j-1}g$ with $g\in C^{j-1}(\Delta)$, then whp also $g\in
  C^{j-1}(\cG_p)$ and thus $f_p+\delta^{j-1}g$ lies in the same
  cohomology class of $H^j(\cG_p;\FF_2)$ as $f_p$. By the minimality
  of $S_p$, this implies that $w(f_p) = |S_p| = s$ whp. For the rest
  of the proof, let us condition on this high probability event.

  Now Proposition~\ref{prop:meshwal} tells us that at least
  $\frac{sn}{j+2}$ many $(j+2)$-sets would form odd $(j+1)$-simplices
  if they were present in $\cG_p$. The fact that $f_p$ is a
  $j$-cocycle implies that no such $(j+2)$-set is allowed to be in a
  $k$-simplex. Each $(j+2)$-set is contained in $\binom{n-j-2}{k-j-1}$
  many $(k+1)$-sets, each of which contains $\binom{k+1}{j+2}$ many 
  $(j+2)$-sets. Therefore the number of $(k+1)$-sets that cannot be
  chosen as $k$-simplices in $\cG_p$ is at least
  \begin{equation*}
    \frac{sn \binom{n-j-2}{k-j-1}}{(j+2) \binom{k+1}{j+2}} \ge
    \alpha_0 s n^{k-j} \geq \alpha_0 \ell n^{k-j},
  \end{equation*}
  for some constant $\alpha_0=\alpha_0(k,j)>0 $. Thus, the probability
  that a fixed support exists together with the $\ell$ many
  $k$-simplices that make it traversable, but that no odd
  $(j+1)$-simplices are present is at most
  \begin{equation*}
    \left(p(1-p)^{\alpha_0 n^{k-j}} \right)^\ell.
  \end{equation*}
  The derivative of this expression with respect to $p$ is negative
  throughout the range $p\ge\pjmin$, therefore in the following
  calculations involving $p$ we can use the lower bound $\pjmin$.
  Given the sequence $\underline{b}$, the probability
  $q_{\underline{b}}$ that some such support exists and that the
  connecting $k$-simplices have no odd $(j+1)$-simplices satisfies
  \begin{align*}
    q_{\underline{b}} \prod_{i=1}^{s} b_i!
    &\le n^{j+1}\left(2^{\binom{k+1}{j+1}}\binom{n}{k-j}p(1-p)^{\alpha_0 n^{k-j}}\right)^\ell\\
    &\le n^{j+1}\left(2^{\binom{k+1}{j+1}}\frac{(j+1)\log n}{k-j+1}e^{-(1-o(1))\alpha_0\frac{j+1}{k-j+1}(k-j)!\log n}\right)^\ell\\
    &\le n^{j+1}\left(n^{-\alpha_0\frac{j}{k-j+1}(k-j)!}\right)^\ell\\
    &\le n^{j+1}n^{-\alpha_1\ell} \le n^{-\frac{\alpha_1}{2}\ell},
  \end{align*} 
  where $\alpha_1 = \alpha_1(k,j)>0$ and the last inequality holds for
  $\ell \ge \frac{2(j+1)}{\alpha_1}$. Moreover, since $\ell \ge
  \frac{s}{\binom{k+1}{j+1}}$, we can find another positive constant
  $\alpha_2$ such that 
  \begin{equation}\label{eq:pb}
    q_{\underline{b}} \prod_{i=1}^{s} b_i! \leq n^{-\alpha_2 s}.
  \end{equation}

  For each sequence $\underline{b}=(b_1,\ldots,b_s)$ define
  \begin{equation*}
    t(\underline{b}) := |\{ i : b_i \geq n^{\alpha_2/2} \} |
  \end{equation*}
  and let $B_t$ be the set of all sequences $\underline{b}$ such that
  $t(\underline{b})=t$. We can crudely bound $|B_t|$, the number of
  sequences in $B_t$, by
  \begin{equation*}
    s^t \binom{n}{k-j}^t (n^{\alpha_2/2})^{s-t}.
  \end{equation*}
  On the other hand, if $\underline{b} \in B_t$, then
  \begin{equation*}
    \prod_{i=1}^{s} b_i!
    \ge \left(\left(n^{\alpha_2/2}\right)!\right)^t
    \ge \left(\left(n^{\alpha_2/2}\right)^{n^{\alpha_2/3}}\right)^t
    \ge n^{t n^{\alpha_2/4}}.
  \end{equation*}
  Summing over all possible sequences $\underline{b}$, we obtain
  \begin{align} \label{eq:sequence}
    \sum_{\underline{b}} \frac{1}{\prod_{i=1}^{s} b_i!} &=
    \sum_{t=0}^s\sum_{\underline{b}\in B_t}\frac{1}{\prod_{i=1}^{s}b_i!} \nonumber\\
    &\le \sum_{t=0}^s\frac{s^t \binom{n}{k-j}^t (n^{\alpha_2/2})^{s-t}}{n^{t n^{\alpha_2/4}}} \nonumber\\
    &= n^{\alpha_2s/2}\sum_{t=0}^{s}\left(\frac{s \binom{n}{k-j}}{n^{\alpha_2/2} n^{n^{\alpha_2/4}} } \right)^t \nonumber\\
    &\le (s+1)n^{\alpha_2s/2}.
  \end{align}
  Combining~\eqref{eq:pb} and~\eqref{eq:sequence}, the probability
  that some support of fixed size $s$ exists is at most
  \begin{equation*}
    (s+1) n^{\alpha_2s/2} n^{-\alpha_2 s} \leq n^{-\alpha_2 s /3}.
  \end{equation*}
  Let $\bar{d}> \frac{4(k+1)}{\alpha_2}$ be a constant. If we sum over
  all $s \geq \bar{d}$, we see that the probability that $S_p$ exists
  and $|S_p| \ge 
  \bar{d}$ is at most $n^{- \alpha_2 \bar{d}/ 4 }$. This holds for
  every $p\geq \pjmin$ and thus, taking a union bound over all
  $O(n^{k+1})$ birth times in this range, the probability for $S_p$ of
  size at least $\bar{d}$ to exist for any $p\ge\pjmin$ is
  $O\left(n^{k+1-(\alpha_2\bar{d}/4)}\right)$, which tends to zero for
  our choice of $\bar{d}$.
\end{proof}
We can now show that whp for $p$ ``close'' to $\pj$ each $j$-cocycle in $\cG_p$ arises from copies of $\Mjmin$.

\begin{cor}\label{cor:noSpsmallp}
For every $p=(1+o(1))\pj$ with $p\geq \pjmin$, we have 
$\mathcal{N}_{\cG_p}=\emptyset$ with high probability.
\end{cor}

\begin{proof}
  By Lemma~\ref{lem:minobst} and the definition of
  $\mathcal{N}_{\cG_p}$ (Definition~\ref{def:generating}), whp either
  $S_p$ does not exist or $|S_p| \ge
  k-j+2$. Furthermore, Lemma~\ref{lem:largesupport} tells us that
  whp for all $p\ge\pjmin$, either $S_p$ does not exist or it must be
  of constant size. For $p=(1+o(1))\pj$, Lemma~\ref{lem:smallsupport}
  implies that whp $S_p$ does \emph{not} have constant size, and thus
  whp $S_p$ does not exist, meaning that whp $\mathcal{N}_{\cG_p}$ is empty.
\end{proof}

\subsection{Monotonicity with high probability} 

Although the existence of bad functions in $\cG_p$ is not
intrinsically a monotone property, in this section we show
that in fact, from time $\pMj$ on, whp this property behaves in 
a monotone way.

By Corollary \ref{cor:pMj}, whp we can apply Corollary \ref{cor:noSpsmallp}
with $p=\pMj$, therefore whp  $\mathcal{N}_{\cG_{\pMj}}$ is 
empty. In other words, whp there are no bad functions in $\cG_{\pMj}$, i.e.\
$H^j(\cG_{\pMj};\FF_2)=0$. However, 
we still need to prove that $\cG_p$ does not lose this
property for any larger $p$. More precisely, we already know by
Lemma~\ref{lem:largesupport} that whp no $\cG_p$ for $p\ge\pMj$
contains a $j$-cocycle with ``large'' support, but ``small'' supports
have been excluded by Lemma~\ref{lem:smallsupport} only in the range
$p=(1+o(1))\pj$. In the next lemma we show that if a new obstruction
appears, then the $k$-simplex whose birth causes this appearance must
be a local obstacle (Definition~\ref{def:localobstacle}). But
we already know by Lemma~\ref{lem:localobstacle} that whp no new local
obstacles appear, which will complete the argument.

\begin{lem} \label{lem:nomoreMj}
  Whp either $\mathcal{N}_{\cG_p} = \emptyset$ for all $p \ge \pMj$
  or the $k$-simplex $K$ with smallest birth time $p_K \ge \pMj$, for
  which $\mathcal{N}_{\cG_{p_K}} \not= \emptyset$, forms a local
  obstacle in $\cG_{p_K}$.
\end{lem}

\begin{proof}
  The lemma is trivially true if whp $\mathcal{N}_{\cG_p} = \emptyset$
  for all $p \ge \pMj$, we may thus assume that $K$ exists with
  positive probability. Let $p < p_K$ be such that $\cG_{p_K} =
  \cG_p+K$.

  Suppose first that $S_{p_K} \cap \cG_p \not= \emptyset$. Let $S$ be
  a maximal subset of $S_{p_K}$ which is traversable in $\cG_p$ and
  let $f$ be the $j$-cochain in $\cG_p$ with support $S$. Every
  $k$-simplex of $\cG_p$ containing some $j$-simplex in $S$ cannot
  contain $j$-simplices in $S_{p_K} \setminus S$ by the maximality of
  $S$. Therefore, every $(j+1)$-simplex of $\cG_p$ is even with
  respect to $f$, because it is even with respect to $f_{p_K}$. This
  means that $f$ is a $j$-cocycle in $\cG_p$.

  Lemma~\ref{lem:largesupport} implies that there exists a constant
  $\bar{d}$ such that whp $|S_{p_K}| < \bar{d}$ and thus also $|S| <
  \bar{d}$. But Lemma~\ref{lem:jshells}, together with the fact
  that $p>\pMj>\pjone$ whp, implies that whp each
  $j$-simplex in $S$ lies in linearly many $j$-shells in $\cG_p$, at
  most $|S|-1$ of which can contain other elements of $S$. Thus, whp
  there are $j$-shells in $\cG_p$ that contain precisely one element
  of $S$, which means that $f$ is not a $j$-coboundary, i.e.\ $f$ is a
  bad function in $\cG_p$. Now recall that whp there are no copies of
  $\Mjmin$ in $\cG_p$ by Corollary~\ref{cor:noMjmin} and thus all bad
  functions lie in $\mathcal{N}_{\cG_p}$. This means that
  $\mathcal{N}_{\cG_p} \not= \emptyset$, a contradiction to the choice
  of $K$.

  Thus, whp $S_{p_K}$ is entirely contained in $K$ and its simplices
  are not in other $k$-simplices of $\cG_{p_K}$. Moreover,
  it follows from 
  Lemma~\ref{lem:minobst} that $|S_{p_K}| \ge k-j+1$, implying
  that whp $K$ forms a local obstacle in $\cG_{p_K}$.
\end{proof}

The following corollary shows that in the supercritical regime 
$p\geq \pMj$, whp no $j$-cocycle arises from copies of $\Mjmin$.

\begin{cor}\label{cor:supercritical}
  With high probability
  $\mathcal{N}_{\cG_p}= \emptyset$ for all $p\ge\pMj$ simultaneously.
\end{cor}

\begin{proof}
  Recall that by Corollaries~\ref{cor:pMj}~and~\ref{cor:noSpsmallp},
  $\mathcal{N}_{\cG_{\pMj}} = \emptyset$ whp. If
  $\mathcal{N}_{\cG_p} \not= \emptyset$ for some $p>\pMj$, then whp the
  $k$-simplex whose birth creates a $j$-cocycle that is not generated
  by copies of $\Mjmin$ would form a local obstacle by
  Lemma~\ref{lem:nomoreMj}. But Lemma~\ref{lem:localobstacle} tells us
  that whp no new local obstacles appear after time $\bar{p}_j$, which
  whp is smaller than $\pMj$ by~\eqref{eq:pjbar} and
  Corollary~\ref{cor:pMj}.
\end{proof}

\section{Proofs of main results} \label{sec:proofs}

\subsection{Proof of
  Theorem~\ref{thm:gentheor}}\label{sec:proofs:gentheor}

Corollary~\ref{cor:pMj} states that for any function $\omega$ of $n$
which tends to infinity as $n \rightarrow \infty$, whp we have 
\begin{equation*}
  \frac{(j+1)\log n +\log \log n -\omega}{(k-j+1)n^{k-j}}(k-j)!
  < \pMj
  < \frac{(j+1)\log n +\log \log n +\omega}{(k-j+1)n^{k-j}}(k-j)!,
\end{equation*} 
which is precisely Theorem~\ref{thm:gentheor}~\ref{thm:gentheor:pMj}.

To prove~\ref{thm:gentheor:subcrit}, recall that
Lemma~\ref{lem:phantomlemma} states that for all $i\in[j]$, whp
$H^i(\cG_p;\FF_2) \neq 0$ for all $p\in[p_{i-1}^-,p_{M_i})$.
By~\ref{thm:gentheor:pMj}, whp for all $i\in[j-1]$
\begin{equation*}
  p_{M_i} > \left( 1 - \frac{1}{\sqrt{\log n}} \right)
  \frac{(i+1)\log n}{(k- i + 1) n^{k-i}} (k-i)! = p_{i}^-,
\end{equation*}
and thus whp $\cG_p$ is not \connected\ throughout
$\bigcup\limits_{i=1}^{j} [p_{i-1}^-,p_{M_i}) = [p_0^-,\pMj)$.

Now observe that by Lemma~\ref{lem:topconn} whp $p_T > p_0^-$ and that
$\cG_p$ is not topologically connected in $[0,p_T)$ by the definition
of $p_T$. Therefore, whp $\cG_p$ is not \connected\ in
\begin{equation*}
  [0, \pMj) = [0, p_T) \cup [p_0^-, \pMj),
\end{equation*}
as required.

It remains to prove~\ref{thm:gentheor:supercrit}. We have to show that
whp there are no bad functions in $\cG_p$ for every $p\geq \pMj$. By
Corollary~\ref{cor:noMjmin}, whp for all $p\ge\pMj$, there are no
copies of $\Mjmin$ in $\cG_p$. Thus, if $H^j(\cG_p;\FF_2)\not=0$, then
any representative of a non-zero cohomology class cannot arise from
copies of $\Mjmin$ and therefore lies in $\mathcal{N}_{\cG_p}$
(Definition~\ref{def:generating}). But by
Corollary~\ref{cor:supercritical}, whp each such $\mathcal{N}_{\cG_p}$
is empty and thus whp $H^j(\cG_p;\FF_2)=0$ for all $p\geq \pMj$.
Analogously, whp all cohomology groups $H^i(\cG_p;\FF_2)$ for
$i\in[j-1]$ vanish, because whp $p_{M_i} < \pMj$
by~\ref{thm:gentheor:pMj}. Finally, by~\ref{thm:gentheor:pMj} and
Lemma~\ref{lem:topconn} whp $p_T < \pMj$, meaning that whp $\cG_p$ is
topologically connected for all $p\ge\pMj$. This implies that whp each
such $\cG_p$ is \formalconnected.
\qed

\subsection{Proof of
  Corollary~\ref{cor:hittingtimeYp}}\label{sec:proofs:corhitting}

Let $\omega$ be any function of $n$ which tends to infinity as
$n\to\infty$. It is known (see e.g.~\cite{MeshulamWallach08}) that whp
\begin{equation} \label{eq:pisolbound}
  \frac{k \log n - \omega}{n} < \pisol < \frac{k \log n + \omega}{n}.
\end{equation}
The proof is an easy application of the first and second moment
methods.

In order to prove that $\pconn = \pisol$ whp, suppose that a
$(k-1)$-simplex $\sigma$ is isolated in $\cY_p$ for some $p$. The
indicator function $f_\sigma$ of $\sigma$ is a $(k-1)$-cocycle,
because $\sigma$ is isolated. But $f_\sigma$ is \emph{not} a
$(k-1)$-coboundary, because $\sigma$ lies in ($n-k$ many)
$(k-1)$-shells. In particular, $H^{k-1}(\cY_p;\FF_2)\neq 0$. By the
definitions of $\pconn$ and $\pisol$, this implies that $\pconn \geq
\pisol$.

For the opposite direction, fix the birth times of all $k$-simplices.
Then for all $p\geq \pisol$, we have $\cY_p = \cG_p$ and therefore
$\cY_p$ is  \mwconnected\ whp for every $p \geq
\max(\pisol,p_{M_{k-1}})$ by
Theorem~\ref{thm:gentheor}~\ref{thm:gentheor:supercrit}.
By~\eqref{eq:pisolbound} and
Theorem~\ref{thm:gentheor}~\ref{thm:gentheor:pMj}, whp for any
(slowly) growing function $\omega$
\begin{equation*}
  \pisol > \frac{k \log n - \omega}{n}
  > \frac{k \log n + \log\log n + \omega}{2n} > p_{M_{k-1}},
\end{equation*}
hence whp for all $p\geq \pisol$ we have $H^{k-1}(\cY_p;\FF_2) =
H^{k-1}(\cG_p;\FF_2) = 0$. This means that whp $\pconn \leq \pisol$
and thus $\pconn = \pisol$, as required.
\qed

\subsection{Proof of
  Theorem~\ref{thm:critwindow1}}\label{sec:proofs:critwind}

We are interested in the asymptotic distribution of $D_j :=
\dim\left(H^j(\cG_p;\FF_2)\right)$ for
\begin{equation*}
  p = \frac{(j+1)\log n+\log\log n+\const_n}{(k-j+1)n^{k-j}}(k-j)!,
\end{equation*}
where $\const_n \xrightarrow{n\rightarrow \infty} \const \in
\mathbb{R}$.

Recall that $\varMjmin$ is the random variable defined in
Lemma~\ref{lem:meanconc} which counts the number of copies of
$\Mjmin$. We  apply the method of moments
(Lemma~\ref{lem:metmoments}) to $\varMjmin$, showing that it
converges in distribution to a Poisson random variable with
expectation
\begin{equation*}
  \lambda_j = \frac{(j+1)e^{-\const}}{(k-j+1)^2 j!}.
\end{equation*}
Subsequently, we will prove that whp $\varMjmin = D_j$. In particular
this will imply that 
\begin{equation*}
  D_j \xrightarrow{d} \mbox{Po}(\lambda_j),
\end{equation*} 
as required. 

In order to determine the expectation of $\varMjmin$, let
$K\subset[n]$ be a $(k+1)$-set and let $C$ be a $j$-subset of $K$.
Recall that the probability that a (potential) petal $C\cup\{w\}$ with $w\in
K\setminus C$ lies in no other $k$-simplex is given by  
\begin{equation*}
  r = (1-p)^{\binom{n-j-1}{k-j} - 1}
\end{equation*}
(see~\eqref{eq:r}). Arguing as in Lemma~\ref{lem:meanconc}, we see that
dependencies between the petals are negligible and thus
\begin{equation}\label{eq:expecmin}
  \EE(\varMjmin) = (1+o(1)) \binom{n}{k+1} \binom{k+1}{j} p r^{k-j+1}.
\end{equation}
We observe that
\begin{align}
  r^{k-j+1}
  &= (1-p)^{\left(\binom{n-j-1}{k-j}-1\right)(k-j+1)}\nonumber\\
  &= \exp\left(-\frac{n^{k-j}}{(k-j)!}(k-j+1)p+O\left(n^{k-j-1}p\right)+O\left(n^{k-j}p^2\right)\right)\nonumber\\
  &= \exp\left(-(j+1)\log n - \log\log n-\const_n+o(1)\right)\nonumber\\
  &= (1+o(1))\frac{e^{-\const_n}}{n^{j+1}\log n}.\label{eq:ro1}
\end{align}
Therefore, we have
\begin{align}
  \EE(\varMjmin) &=(1+o(1)) \frac{n^{k+1}}{(k-j+1)!j!} \cdot
  \frac{(j+1)\log n+\log\log n+\const_n}{(k-j+1)n^{k+1}\log n}(k-j)!e^{-\const_n}\nonumber\\
  &= (1+o(1))\frac{(j+1)e^{-\const_n}}{(k-j+1)^2j!}
  \stackrel{\const_n\to\const}{=} (1+o(1))\lambda_j.\label{eq:lambda}
\end{align}

Denote by $\mathcal{T}^-$ the set of all pairs $(K,C)$ that can form a
copy of $\Mjmin$ in $\cG_p$. For each $T^-\in\mathcal{T}^-$, denote by
$X_{T^-}$ the indicator random variable of the event that $T^-$ forms
a copy of $\Mjmin$ in $\cG_p$. For each fixed integer $t\geq 1$, we
now determine the binomial moments
\begin{equation*}
  \EE\binom{\varMjmin}{t} =
  \sum_{\mathcal{S}\in\binom{\mathcal{T}^-}{t}}
  \Pr\left(\bigcap_{T^-\in\mathcal{S}} \{ X_{T^-}=1 \} \right).
\end{equation*}

Suppose first that all $T^-\in\mathcal{S}$ have different
$(k+1)$-sets. In this case, if all $T^-\in\mathcal{S}$ form copies of
$\Mjmin$, none of the petals are shared (by
property~\ref{Mjmin:flower} of $\Mjmin$, see
Definition~\ref{def:mjminus}). If we choose $t$ distinct $(k+1)$-sets
uniformly at random, whp they will be disjoint and in particular no
two $T_1^-,T_2^-\in\mathcal{S}$ will share a petal. To choose $t$
distinct $(k+1)$-sets, there are
\begin{equation*}
  \binom{\binom{n}{k+1}}{t} = (1+o(1))\frac{\binom{n}{k+1}^t}{t!}
\end{equation*}
choices. Therefore, the contribution to $\EE \binom{\varMjmin}{t}$ made
by the sets $\mathcal{S}$ for which all
$T^-\in\mathcal{S}$ have distinct $(k+1)$-set is
\begin{align}
  (1+o(1))\binom{\binom{n}{k+1}}{t} \binom{k+1}{j}^t p^t r^{t(k-j+1)}
  &\stackrel{\eqref{eq:expecmin}}{=} (1+o(1))\frac{\EE(\varMjmin)^t}{t!}\nonumber\\ 
  &\stackrel{\eqref{eq:lambda}}{=} (1+o(1))\frac{{\lambda_j}^t}{t!},\label{eq:critwindcontr}
\end{align} 
which is the desired asymptotic value.

We now show that the contribution coming from sets $\mathcal{S}$ whose
elements use $u<t$ different $(k+1)$-sets is negligible. We have
$\binom{\binom{n}{k+1}}{u}$ ways to select the $(k+1)$-sets and at
most $u^{t-u}\binom{k+1}{j}^t$ different ways to locate the $t$
potential $\Mjmin$ in them. Moreover, observe that two different
copies of $\Mjmin$ in the same $k$-simplex share at most one petal
(otherwise they would have the same centre and thus be identical) and
in that case these two copies have $(k-j+1) + (k-j)$ petals in total.
This means that each of the $u$ many $(k+1)$-sets contains at least
$k-j+1$ petals, and at least one $(k+1)$-set contains at least
$(k-j+1) + (k-j)$ petals. Therefore the total number of petals
required for such a set $\mathcal{S}$ is bounded from below by
$u(k-j+1) + (k-j)$. In total, the contribution of
such sets $\mathcal{S}$ to the binomial moment is at most
\begin{equation*}
  \binom{\binom{n}{k+1}}{u} u^{t-u} \binom{k+1}{j}^t p^u r^{u(k-j+1)}
  r^{k-j}.
\end{equation*}
Replacing $t$ by $u$ in~\eqref{eq:critwindcontr}, we deduce that
\begin{equation*}
  \binom{\binom{n}{k+1}}{u} u^{t-u} \binom{k+1}{j}^t p^u r^{u(k-j+1)}
  = (1+o(1))\frac{{\lambda_j}^u}{u!}\left(u\binom{k+1}{j}\right)^{t-u}
  = \Theta(1).
\end{equation*}
Furthermore,~\eqref{eq:ro1} yields $r^{k-j} =
o(1)$. Together with~\eqref{eq:critwindcontr}, we deduce that
\begin{equation*}
  \EE\binom{\varMjmin}{t} = (1+o(1))\frac{\lambda_j^t}{t!}
\end{equation*}
for each fixed integer $t\ge1$. Now Lemma~\ref{lem:metmoments}
yields $\varMjmin \xrightarrow{d} \mbox{Po}(\lambda_j)$.

It remains to show that $\varMjmin=D_j$ whp. To this end, denote by
$f_1,\ldots,f_{\varMjmin}$ the $j$-cocycles arising from the copies of
$\Mjmin$ in $\cG_p$. Corollary~\ref{cor:noSpsmallp} in particular
implies that whp the cohomology classes of $f_1,\ldots, f_{\varMjmin}$ 
generate $H^j(\cG_p;\FF_2)$, which means that $\varMjmin\geq D_j$.  

In order to prove the opposite direction, we show that the cohomology
classes of $f_1,\ldots,f_{\varMjmin}$ are linearly independent.
Observe first that whp $\varMjmin = o(n)$ by Markov's inequality,
because $\varMjmin$ has bounded expectation. Let $I\subseteq
[\varMjmin]$ be non-empty and let $S$ be the support of
$\sum_{i\in I}f_i$. By the arguments above for $t=2$ and $u=1$, whp
there are no two $\Mjmin$ that share the same $k$-simplex. Thus, whp
the $f_i$'s have disjoint support by property~\ref{Mjmin:flower} of an
$\Mjmin$ (Definition~\ref{def:mjminus}), and in particular $S \neq 
\emptyset$. Pick $L \in S$. Lemma~\ref{lem:jshells} and the
fact that $p>\pjone$ tell us that whp
there are $\Theta(n)$ many $j$-shells in $\cG_p$ that contain $L$. All
these $j$-shells meet only in $L$, thus at most $|S|\leq (k-j+1)|I| =
o(n)$ of them can contain another $j$-simplex in $S$. Thus, there are
$j$-shells that meet $S$ only in $L$, showing that $\sum_{i\in I} f_i$
is not a $j$-coboundary. Therefore the cohomology classes of
$f_1,\dotsc,f_{\varMjmin}$ are linearly independent whp. This shows
that whp $\varMjmin \le D_j$ and thus $\varMjmin=D_j$, as desired. 

Together with $\varMjmin \xrightarrow{d} \mbox{Po}(\lambda_j)$, this
proves that $D_j \xrightarrow{d} \mbox{Po}(\lambda_j)$. By
Theorem~\ref{thm:gentheor} (for $j-1$ instead of $j$) whp
$H^0(\cG_p;\FF_2)=\FF_2$ and $H^i(\cG_p;\FF_2)=0$ for all $i\in[j-1]$.
In particular,
\begin{align*}
  \Pr(\cG_p \text{ is \connected})
  &= \Pr\big(H^j(\cG_p;\FF_2)=0\big)+o(1)\\
  &= (1+o(1))\Pr\big(\mbox{Po}(\lambda_j)=0\big)\\
  &= (1+o(1))e^{-\lambda_j}.
\end{align*}
This concludes the proof of Theorem~\ref{thm:critwindow1}.
\qed

\section{Concluding remarks} \label{sec:concremarks}

\subsection{Comparison of proof methods} \label{sec:concremarks:dim}

Let us note that for the subcritical regime
(Theorem~\ref{thm:gentheor}~\ref{thm:gentheor:subcrit}), one might try
to use a different approach in order to prove that $H^j(\cG_p;\FF_2)$
does not vanish in the interval $[\pjmom,\pMj)$. If the dimension of 
$C^j(\cG_p)$ (viewed as an $\FF_2$-vector space) is larger than the
sum of the dimensions of $C^{j-1}(\cG_p)$ and $C^{j+1}(\cG_p)$, then
$H^j(\cG_p;\FF_2)\neq 0$ would follow. However, this behaviour only
happens for ``small'' $p\in [\pjmom,\pMj)$ and, more importantly, only
for $j \geq \frac{k-1}{2}$. In contrast, our proof method works for
all values of $j$. Moreover, our result that $[\pjmom,\pMj)$ whp is
covered by three copies of $\Mj$ (Lemma~\ref{lem:phantomlemma}),
together with the fact that whp $\cG_{p_0^-}$ has isolated
vertices (this can be proved using an easy second moment argument),
implies the following slightly stronger statement.
\begin{prop}
  With high probability for every $p<\pMj$, the complex $\cG_p$
  contains an isolated vertex or a copy of $M_i$ for some $i \in [j]$.
\end{prop}

In the supercritical regime
(Theorem~\ref{thm:gentheor}~\ref{thm:gentheor:supercrit}), the
counting methods used in~\cite{LinialMeshulam06,MeshulamWallach08} for
$\cY_p$ are \emph{not} sufficient to prove the non-existence of
$j$-cocycles in $\cG_p$. This is due to the fact that these methods
have been designed for the special case $j=k-1$ and for a threshold
which is about twice as large as $p_{k-1}$. For this reason, the more
careful arguments used in Lemmas~\ref{lem:traversable}
to~\ref{lem:nomoreMj} become necessary.

\subsection{Alternative models}

There are several ways to define random $k$-complexes. If the
$k$-simplices are chosen independently with probability $p$, then the
models $\cY_p$ and $\cG_p$ are somewhat extremal constructions, in the
sense that $\cY_p$ contains \emph{all} simplices of lower dimension,
while $\cG_p$ only comprises those simplices that are necessary in
order to be a complex. What happens in between, i.e.\ when the complex
contains all simplices in $\cG_p$, but in addition, some simplices of
dimensions $1,\ldots,k-1$ might be added in a random fashion?
Depending on the choice of probabilities, such a complex might show
behaviour that is different from both $\cY_p$ and $\cG_p$.

Random complexes also arise naturally from random graphs. For
instance, the random \emph{clique complex} $\mathcal{X}_p(n)$ (also
known as \emph{flag complex}) on vertex set $[n]$ can be defined as
the maximal complex whose 1-skeleton is the binomial random graph.
Equivalently, a non-empty set $U \subseteq [n]$ forms a simplex in
$\mathcal{X}_p(n)$ if and only if $U$ is a clique in the binomial
random graph. Topological properties of $\mathcal{X}_p(n)$ have been
studied in~\cite{CostaFarberHorak15,Kahle09,Kahle14}. Another example
is the random \emph{neighbourhood complex} arising from the binomial
random graph by letting each non-empty set of vertices that have a
common neighbour form a simplex~\cite{Kahle07}.
See~\cite{KahleSurvey14} for an overview of these and other models of
random complexes.

\subsection{Other notions of connectedness}

The vanishing of cohomology groups with coefficients in $\FF_2$ is
just one possible way of defining the concept of ``connectedness'' of
$\cG_p$. An obvious alternative would be to consider coefficients from
other groups or fields. For $\cY_p$, such notions of connectedness
have been studied for coefficients in any finite abelian group, in
$\mathbb{Z}$, or in any field
\cite{AronshtamLinial15,AronshtamLinialLuczakMeshulam12,HoffmanKahlePaquette17,LinialPeled16,LuczakPeled16,MeshulamWallach08}.
In particular, the threshold for the vanishing of $H^{k-1}(\cY_p;R)$
for a finite abelian group $R$ is independent of the choice of
$R$~\cite{MeshulamWallach08}. 

For $\cG_p$, it is not obvious whether the threshold for
\connectedness\ depends on the choice of the group of coefficients. An
indication that it might indeed depend on the group, even if we
restrict attention only to finite abelian groups, is the observation
that $\Mj$ only remains an obstruction when the coefficients are taken
from a group of even order. For groups of odd order, the minimal
obstruction becomes larger, and thus one would expect the threshold
for \connectedness\ to decrease. 

A rather strong notion of connectedness would be to require the
homotopy groups $\pi_1(\cG_p), \ldots, \pi_j(\cG_p)$ to vanish. For
the $2$-dimensional case, the vanishing of $\pi_1(\cY_p)$ was studied
by Babson, Hoffman and Kahle~\cite{BabsonHoffmanKahle11}. In
particular, they showed that whp $\pi_1(\cY_p)\neq 0$ at the time that
$H^1(\cY_p;\FF_2)$ becomes zero. From that time on, the models $\cY_p$
and $\cG_p$ coincide. As $\pi_1(\cG_p) \neq 0$ follows immediately
from $H^1(\cG_p;\FF_2) \neq 0$, the range that should be of particular
interest with respect to $\pi_1(\cG_p)$ in the $2$-dimensional case is
\begin{equation*}
  \frac{\log n+\frac{1}{2}\log \log n}{n} \leq p \leq
  \frac{2 \log n + \omega}{n}.
\end{equation*}
A natural conjecture would be that whp $\pi_1(\cG_p) \neq 0$ in this
range.

Theorem~\ref{thm:critwindow1} provides a limit result for the
dimension $D_j = \dim(H^j(\cG_p;\FF_2))$ of the $j$-th cohomology
group of $\cG_p$ around the point of the phase transition. It would be
interesting to know the behaviour of $D_j$ also for earlier regimes.
More precisely, how large is $D_j$ in the interval $[p_{j-1}^-,\pMj)$?
How far below $p_{j-1}^-$ do we have $D_j>0$ whp?

\section*{Acknowledgement}

The authors thank Penny Haxell for very helpful discussions regarding
the two-dimensional case. An extended abstract of the case $k=2$ has
appeared in the Proceedings of
Eurocomb 2017~\cite{CooleyHaxellKangSpruessel17}.

\bibliographystyle{amsplain}
\bibliography{References}

\end{document}